\documentclass[11pt]{amsart}
\usepackage{amssymb}
\usepackage{amsmath}
\usepackage{amsfonts}
\usepackage{graphicx}

\usepackage{hyperref}
    \usepackage{aeguill}
    \usepackage{type1cm}

\theoremstyle{plain}
\newtheorem{thm}{Theorem}[section]

\newtheorem{claim}[thm]{Claim}

\newtheorem{corollary}[thm]{Corollary}

\newtheorem{definition}[thm]{Definition}
\newtheorem{example}[thm]{Example}

\newtheorem{lemma}[thm]{Lemma}

\newtheorem{proposition}[thm]{Proposition}
\newtheorem{remark}[thm]{Remark}

\newtheorem{theorem}[thm]{Theorem}
\numberwithin{equation}{section}
\newcommand{\N}{\mathbb{N}}

\newcommand{\R}{\mathbb{R}}

\begin{document}

\title[Regularization by Lasry-Lions method on manifolds]{Regularization by sup-inf convolutions on Riemannian manifolds: an extension of Lasry-Lions theorem to manifolds of bounded curvature}
\author{D. Azagra}
\address{ICMAT (CSIC-UAM-UC3-UCM), Departamento de An{\'a}lisis Matem{\'a}tico,
Facultad Ciencias Matem{\'a}ticas, Universidad Complutense, 28040, Madrid, Spain}
\email{azagra@mat.ucm.es}

\author{J. Ferrera}
\address{Departamento de An{\'a}lisis Matem{\'a}tico,
Facultad Ciencias Matem{\'a}ticas, Universidad Complutense, 28040, Madrid, Spain}
\email{ferrera@mat.ucm.es}

\date{January 20, 2014}



\begin{abstract}
We show how Lasry-Lions's result on regularization of functions defined on $\R^n$ or on Hilbert spaces by sup-inf convolutions with squares of distances can be extended to (finite or infinite dimensional) Riemannian manifolds $M$ of bounded sectional curvature. More specifically, among other things we show that if the sectional curvature $K$ of $M$ satisfies $-K_0\leq K\leq K_0$ on $M$ for some $K_0>0$, and if the injectivity and convexity radii of $M$ are strictly positive, then every bounded, uniformly continuous function $f:M\to\R$ can be uniformly approximated by globally $C^{1,1}$ functions defined by
$$
(f_{\lambda})^{\mu}=\sup_{z\in M}\inf_{y\in M}\{f(y)+\frac{1}{2\lambda} d(z,y)^{2}-\frac{1}{2\mu}d(x,z)^2\}
$$
as $\lambda, \mu\to 0^{+}$, with $0<\mu<\lambda/2$.
Our definition of (global) $C^{1,1}$ smoothness is intrinsic and natural, and it
reduces to the usual one in flat spaces, but we warn the reader that, in the noncompact
case, this definition differs from other notions of (rather local) $C^{1,1}$ smoothness that have been
recently used, for instance, by A. Fathi and P. Bernard (based on charts).

The importance of this regularization method lies (rather than on the degree of smoothness
obtained) on the fact that the correspondence $f\mapsto (f_{\lambda})^{\mu}$ is explicit
and preserves many significant geometrical properties that the given functions $f$ may have, such as invariance by a set of isometries, infima,
sets of minimizers, ordering, local or global Lipschitzness, and (only when one additionally
assumes that $K\leq 0$) local or global convexity.

We also give two examples showing that this result completely fails, even for (nonflat)
Cartan-Hadamard manifolds, whenever $f$ or $K$ are not bounded.
\end{abstract}

\maketitle

\section{Introduction and main results}

\noindent Throughout the paper, for a function
$f:M\to\mathbb{R}\cup\{+\infty\}$, we
 define
    $$
    f_{\lambda}(x)=\inf_{y\in M}\{f(y)+\frac{1}{2\lambda}d(x,y)^{2}\}.
    $$
Similarly, for a function $g:M\to\R\cup\{-\infty\}$ we define
$$
g^{\mu}(x)=\sup_{y\in M}\{f(y)-\frac{1}{2\mu}d(x,y)^{2}\}.
$$
Observe that $g^{\mu}=-(-g)_{\mu}$, and therefore all properties of functions of the form
$f_{\lambda}$ have an obvious analogue for functions of the form $f^{\mu}$.
In \cite{Lasry-Lions}, J.-M. Lasry and P.-L. Lions proved that, if $M=E=\R^n$ or a
Hilbert space, and if $f:E\to\R$ is bounded and uniformly continuous,
then the functions $(f_{\lambda})^{\mu}$ are of class $C^{1,1}(E)$ and converge to $f$ uniformly on
$E$ as $\lambda, \mu\to 0^{+}$. The importance of this regularization method lies on the fact that the correspondence $f\mapsto (f_{\lambda})^{\mu}$ is explicit
and preserves many significant geometrical properties that the given functions $f$ may have, such as invariance by a set of isometries, infima,
sets of minimizers, ordering, local or global Lipschitzness, and local or global convexity.
These facts make this regularization method an invaluable tool in optimization, nonsmooth analysis, and many
other areas of pure and applied mathematics.
Lasry-Lions' regularization technique has also very strong connections with PDE theory, through the Lax-Oleinik semigroup of a
Hamilton-Jacobi equation. In fact the functions $u(\lambda, x)=f_{\lambda}(x)$ (respectively
$v(\mu, x)=h^{\mu}(x)$) are the viscosity solutions of the equations
$\frac{\partial u}{\partial \lambda}+
\frac{1}{2}\|\nabla u\|^{2}=0$ on $\R^{+}\times E$ with initial data $u(0, x)=f(x)$ (resp.
$\frac{\partial v}{\partial \mu}-
\frac{1}{2}\|\nabla v\|^{2}=0$ on $\R^{+}\times E$, with initial data $v(0, x)=h(x)$).

It is natural to ask whether Lasry-Lions' theorem remains true in the Riemannian setting, as its potential
applications would also be significant in this field.
It is by now known that the Lasry-Lions Theorem is true for compact Riemannian manifolds in
a more general form (for Lax-Oleinik semigroups associated to Hamilton-Jacobi equations),
see \cite{Fathi, Bernard, FathisBook}, although the optimal Lipschitz constants of the gradients
$\nabla (f_{\lambda})^{\mu}$ do not seem to have been found. The proofs of \cite{Fathi, Bernard, FathisBook} rely on compactness arguments that cannot be extended to noncompact manifolds.
What is more surprising, in the literature there does not seem to be a definition of global $C^{1,1}$ smoothness which
makes sense for noncompact manifolds and has the usual properties that one should expect of such a notion. Fathi's definition in \cite{FathisBook} is only for
locally $C^{1,1}$ functions (a function $f:M\to\R$ is locally $C^{1,1}$ provided $f$ is $C^{1,1}$ when
looked at in charts). In \cite{Fathi} a pointwise Lipschitz constant is introduced by means
of a metric in the tangent bundle, but this notion has the disadvantage that, for instance when one endows $TM$ with Sasaki's metric, there are no
Lipschitz gradients with Lipschitz constant less than $1$, which is unpleasant, as for a
function $f\in C^{2}(M)$ we should expect that the Hessian of
$f$ controls the Lipschitz constant of the gradient $\nabla f$, namely that
$\textrm{Lip}(\nabla f)=\sup_{x\in M}\|D^{2}f(x)\|$.
On the other hand, if one tries to extend Bernard's definition of
$C^{1,1}$ smoothness from compact manifolds \cite{Bernard} to noncompact manifolds, then one
obtains different classes of global $C^{1,1}$ functions, depending on the atlases one uses. And, even in the compact case, {\em the} Lipschitz constant of a gradient $\nabla f$ cannot be
defined through charts (unless one exclusively uses very special charts, like the exponential ones, see Theorem \ref{characterizations of Lipschitz gradients} below).

In this paper we present an intrinsic definition of global $C^{1,1}$ smoothness which
makes sense for every Riemannian manifold, reduces
to the usual one in flat spaces, gives rise to the same class of $C^{1,1}$ functions in the
compact case as Fathi's and Bernard's definitions, allows one to deal with sharp Lipschitz
constants of gradients, and meets most, if not all, of the expectations one
may have about a reasonable definition of global $C^{1,1}$ smoothness. See Definitions
\ref{definition of Lipschitzness of a gradient} and \ref{definiton of semiconvex function},
and Theorem \ref{characterizations of Lipschitz gradients} below.

Returning to the extension of Lasry-Lions regularization technique to Riemannian manifolds,
our main result is the following.

\begin{theorem}\label{main theorem}
Let $M$ be a Riemannian manifold (possibly infinite dimensional) with sectional
curvature $K$ such that $-K_0\leq K\leq K_0$ for some $K_0\geq 0$, and such that the injectivity
and convexity radii of $M$ are strictly positive. Let $f:M\to\R$ be
uniformly continuous and bounded, and $q>1$. Then there exists
$\lambda_0=\lambda(K_0, q, f)>0$ such
that for every $\lambda\in (0, \lambda_0]$ and every
$\mu\in (0, \lambda/2q]$ the regularizations
$(f_{\lambda})^{\mu}$ are uniformly locally $\frac{q}{2\mu}$-semiconvex
and uniformly locally $\frac{q}{2\mu}$-semiconcave, and they
converge to $f$, uniformly on $M$, as $\lambda, \mu\to 0$.

In particular we have that $(f_{\lambda})^{\mu}\in C^{1,1}(M)$ for every such
$\lambda, \mu$. Moreover, we have the following estimations of the Lipschitz constants
of $\nabla\left( (f_{\lambda})^{\mu}\right)$:
$$
\textrm{Lip}\left(\nabla\left( (f_{\lambda})^{\mu}\right)\right)\leq \frac{q}{\mu} \textrm{ if
$M$ is finite dimensional, and}
$$
$$
\textrm{Lip}\left(\nabla\left( (f_{\lambda})^{\mu}\right)\right)\leq 6\frac{q}{\mu} \textrm{ if
$M$ is infinite dimensional.}
$$
Finally, if $f$ is Lipchitz then so is $(f_{\lambda})^{\mu}$, and we have
$$
\lim_{\lambda, \mu\to 0^{+}}\textrm{Lip}\left( (f_{\lambda})^{\mu}\right)=\textrm{Lip}(f).
$$
\end{theorem}

In section 8 we give two examples showing that this result fails (even on a Cartan Hadamard
manifold) if $f$ or $K$ are not bounded. In particular it is clear that the results claimed
without proof in \cite{AnguloVelasco} for Cartan-Hadamard manifolds are totally wrong.

Nevertheless, even if $f$ or $K$ are not bounded,
if one assumes that $K$ is bounded on bounded subsets $B$ of $M$ (which is always the case if
$M$ is complete and finite dimensional), and that $f$ is quadratically minorized on $M$ and
uniformly continuous on bounded subsets of $M$, then the convergence of the functions
$(f_{\lambda})^{\mu}$ to $f$ is uniform on
bounded sets $B$ of $M$, and these functions are of class $C^{1,1}(B)$ for
sufficiently small $\lambda, \mu$ depending on $B$. Of course, in this case one has in general
that
$\textrm{Lip}\left(\nabla(f_{\lambda})^{\mu}_{|_{B}}\right)\to\infty$  as $B$ grows
large.

It is about time we explained what we mean by a $C^{1,1}$ function. If $U$ is an open subset of $\R^n$ or
a Hilbert space and $f:U\to\R$, saying that $f\in C^{1,1}(U)$ just
means that $f\in C^{1}(U)$ and the gradient $\nabla f$ is a Lipschitz mapping from $U$
into $\R^n$, that is, there exists $C\geq 0$ such that $\|\nabla f(x)-\nabla f(y)\|\leq C\|x-y\|$
for every $x,y\in U$. One says that $C$ is a Lipschitz constant for $\nabla f$, and the infimum
of all such $C$ is denoted by $\textrm{Lip}(f)$. The extension of this definition to the Riemannian setting is not
an obvious matter, since for a $C^1$ function $f:M\to\R$ the vectors $\nabla f(x)$ and $\nabla f(y)$
belong to different fibres of $TM$ and in general there is no global way to compare them that
serves all purposes one may have in mind. If one looks for an intrinsic definition
of $C^{1,1}$ smoothness, a natural attempt is to use a metric on $TM$. One can even define pointwise Lipschitz constants of gradients using metrics in $TM$, as Fathi  did in \cite{Fathi}: $$\textrm{Lip}_{x}(\nabla f)=\limsup_{y, z\to x}\frac{d_{TM}(\nabla f(y), \nabla f(z))}{d_{M}(y,z)}.$$
One may then set $\textrm{Lip}(\nabla f)=\sup_{x\in M}\textrm{Lip}_x(\nabla f)$. This leads to declaring a function $f\in C^{1}(M)$ to be of class $C^{1,1}(M)$ provided that the mapping $\nabla f:M\to TM$ is Lipschitz (with respect to the given metrics in $M$ and $TM$). Such a notion of $C^{1,1}$ smoothness can be practical in several ways, but, as we mentioned before, it has the
disadvantage that $\textrm{Lip}_x(f)$ is not finely controlled by the Hessian $D^{2}f(x)$ when $f\in C^{2}(M)$. Indeed, for any Riemannian manifold $M$, if one endows $TM$ with the Sasaki metric (see \cite{Sasaki, Sakai} for the precise definition), since the parallel translation of the zero vector along a geodesic of $M$ is always a geodesic in $TM$, one obtains, for every constant function $c$ on $M$, that $\nabla c(x)=0$ for every $x\in M$, hence also $D^{2}c(x)=0$ for every $x\in M$, and yet $\textrm{Lip}(\nabla c)=1$. Therefore, if one should use this definition of Lipschitzness for gradients, then one would not be able to relate the Lipschitz constants of $\nabla f$ with the semiconvexity and semiconcavity constants of $f$. This is the main reason why we will discard this definition in this paper.

Let us now present our definition of $C^{1,1}$ smoothness.
Let $M$ be a Riemannian manifold (possibly infinite dimensional). We will denote the injectivity radius of $M$ at a point
$x$ by $i(x)$, and the convexity radius of $M$ at $x$ by $c(x)$. We will also denote
$i(M)=\inf_{x\in M}i(x)$, and $c(M)=\inf_{x\in M}c(x)$. It is well known that $i(x)>0$ and
$c(x)>0$ for every $x\in M$ (but $i(M)$ and $c(M)$ may be zero).
Thus, for every $x_0\in M$ there exists $R>0$ such that the ball
$B(x_0, 2R)$ is convex and $\exp_{x}: B_{T_xM}(0,R)\to B(0,R)$ is a $C^{\infty}$ diffeomorphism
for every $x\in B(x_0, R)$. If $x, y\in B(x_0, R)$, let us denote by $L_{xy}: T_xM\to T_yM$ the linear isometry between these tangent spaces provided by parallel translation of vectors along the unique minimizing geodesic connecting the points $x$ and $y$.
More precisely, if $\gamma:[0, \ell]\to M$ is the unique geodesic with
$\gamma(0)=x$, $\gamma(\ell)=y$, $\ell =d(x,y)$, $h\in T_xM$, and $P:[0, \ell]\to TM$ is the unique parallel vector field along $\gamma$ with $P(0)=h$, then we define
$
L_{xy}(h)=P(\ell).
$
When $i(M), c(M)>0$, the isometry $L_{xy}:T_xM\to T_yM$ allows us to compare vectors (or covectors) which are in
different fibers of $TM$ (or $T^{*}M$), in a natural, semiglobal way. Even when the global injectivity
or convexity radii of $M$ vanish, the following definition still makes sense.
\begin{definition}\label{definition of Lipschitzness of a gradient}
{\em Let $M$ be a Riemannian manifold. We say that a function $f:M\to\R$ is of class
$C^{1,1}(M)$ provided $f\in C^{1}(M)$ and there exists $C\geq 0$ such that for every
$x_0\in M$ there exists $r\in \left(0, \min\{i(x_0), c(x_0)\}\right)$ such that $$\|\nabla f(x)-L_{yx}\nabla f(y)\|\leq Cd(x,y)$$ for every $x,y\in B(x_0, r)$. We call $C$ a Lipschitz constant of $\nabla f$. We also say that $\nabla{f}$ is $C$-Lipschitz, and define $\textrm{Lip}(\nabla f)$ as the infimum of all such $C$.}
\end{definition}
It should be noted that, when $i(M), c(M)>0$, this definition is equivalent to the (apparently stronger) following  one: $\|\nabla f(x)- L_{yx}\nabla f(y)\|\leq C d(x,y)$ for every $x, y$ with
$d(x,y)<\min\{ i(M), c(M)\}$.

As is well known in the Euclidean case, $C^{1,1}$ smoothness has much to do with semiconcavity
and semiconvexity of functions, and in the general Riemannian setting we should also expect to find
a strong connection between these notions. Our definition is also satisfactory in this respect,
as we will see soon, but let us first explain what we mean by semiconvex and semiconcave functions. Recall that a function $f:M\to\R$ is said to be convex provided $f\circ\gamma$
is convex on the interval $I\subseteq\R$ for every geodesic segment $\gamma:I\to M$. A function $h$ is called concave if $-h$
is convex.

\begin{definition}\label{definiton of semiconvex function}
{\em Let $M$ be a Riemannian manifold. We will say that a function
$f:M\to (-\infty, +\infty]$ is (globally) {\em semiconvex} if there exists $C>0$ such that
for every $x_0\in M$ the function $M\ni x\mapsto f(x)+C d(x,x_0)^{2}$ is convex.
Similarly, we say that $h:M\to [-\infty, +\infty)$ is (globally) {\em semiconcave} if
there exists $C>0$ such that $h-C d(\cdot, x_0)^{2} $ is concave on $M$, for every $x_0\in M$. Equivalently, $h$ is semiconcave if and only if $-h$ is semiconvex.

We will say that $f$ is {\em locally semiconvex} (resp. {\em locally semiconcave}) if for every $x\in M$ there exists $r>0$ such that $f_{|_{B(x,r)}}: B(x,r)\to [-\infty, +\infty]$ is semiconvex (resp. semiconcave). If there exists $C\geq 0$ such that for every $x_0\in M$ there exists $r>0$ such that the function $B(x_0, r)\ni x\mapsto f(x)+C d(x,y_0)^{2}$ is convex for every $y_0\in B(x_0, r)$, then we will say that $f$ is {\em locally $C$-semiconvex}. We define {\em local $C$-semiconcavity} of a function in a similar way.

Finally, we will say that $f:M\to [-\infty, +\infty]$ is {\em uniformly locally semiconvex} (resp. {\em uniformly locally semiconcave}) provided that there exist numbers $C, R>0$ such that for every $x_0\in M$ the function $$
B(x_0, R)\ni x\mapsto f(x)+C d(x, x_0)^{2}$$ is convex (resp. concave). We will call $C$ a constant of uniformly local semiconvexity (resp. semiconcavity). We will also say that $f$ is uniformly locally $C$-semiconvex (resp, $C$-semiconcave).

It is clear that "globally semiconvex" $\implies$ "uniformly locally semiconvex" $\implies$ "locally semiconvex". }
\end{definition}
\begin{remark}
{\em In the case of a space of constant sectional curvature equal to $0$, in the definition we have just given we could have replaced the condition "there exists $C>0$ such that for every $x_0$ ..." with "there exist $C>0$ and $x_0$ such that ...", and the two definitions would have been equivalent. However, for spaces with nonzero curvature, such two definitions are not equivalent in general. For instance, if $M=\mathbb{H}^{n}$ is the hyperbolic space, the function $x\mapsto d(x,x_0)^{2}$ is $C^{\infty}$ everywhere and the norm of its Hessian  goes to $\infty$ as $d(x,x_0)\to\infty$. Hence this convex function (which would obviously have been semiconcave had we opted for the second
definition) does not belong to $C^{1,1}(\mathbb{H}^{n})$ (see Definition \ref{definition of Lipschitzness of a gradient}, Theorem \ref{characterizations of Lipschitz gradients} and Example \ref{counterexample on hyperbolic space} below). The reason for our choice is that we want a semiconcave and semiconvex function to be of class $C^{1,1}$, as it happens when the function is defined on $\R^n$ or the Hilbert space.}
\end{remark}

That Definition \ref{definition of Lipschitzness of a gradient} is quite satisfactory is clear
from the following.

\begin{theorem}\label{characterizations of Lipschitz gradients}
Let $M$ be a finite dimensional Riemannian manifold, $f\in C^{1}(M,\R)$, and $C\geq 0$. The following statements are equivalent:
\begin{enumerate}
\item $\nabla f$ is $C$-Lipschitz according to Definition \ref{definition of Lipschitzness of a gradient}.
\item For every $x\in M, v\in T_{x}M$ with $\|v\|=1$,
$$
\limsup_{t\to 0^{+}}\frac{1}{t}\|\nabla f(x)-L_{\exp_{x}(tv)x}\nabla f(\exp_{x}(tv))\|\leq C.
$$
\item For every $x_0\in M$ and $\varepsilon>0$ there exists $r>0$ such that
$$
|f\left(\exp_{x}(v)\right)-f(x)-\langle \nabla f(x), v\rangle |\leq \frac{C+\varepsilon}{2}\|v\|^{2}
$$
for every $x\in B(x_0, r)$ and $v\in B_{T_xM}(0,r)$.
\item For every $C'>C$ the function $f$ is locally $\frac{C'}{2}$-semiconvex and locally $\frac{C'}{2}$-semiconcave.
\item For every $x\in M$ and every $\varepsilon>0$ there exists $r>0$ such that, if $F:=f\circ \exp_{x}:B(0,r)\to\R$, then
    $$
    \|\nabla F(u)-\nabla F(v)\|\leq (C+\varepsilon)\|u-v\|
    $$
for every $u,v\in B_{T_xM}(0,r)$.
\item For every $x\in M$ and every $\varepsilon>0$ there exists $r>0$ such that, if $F:=f\circ \exp_{x}:B(0,r)\to\R$, then
    $$
    \|\nabla F(u)-\nabla F(0)\|\leq (C+\varepsilon)\|u\|
    $$
for every $u\in B_{T_xM}(0,r)$.
\end{enumerate}
Moreover, if $f\in C^{2}(M,\R)$ then any of the above statements is also equivalent to the following estimate for the Hessian of $f$:
\begin{enumerate}
\item[{(7)}] $\|D^2 f\|\leq C$.
\end{enumerate}
Finally, if $M$ is of bounded sectional curvature with
$i(M), c(M)>0$, any of the conditions $(1)-(6)$ is equivalent to
\begin{enumerate}
\item[{(4')}] For every $C'>C$ the function $f$ is uniformly locally $\frac{C'}{2}$-semiconvex
and uniformly locally $\frac{C'}{2}$-semiconcave,
\end{enumerate}
and also to
\begin{enumerate}
\item[{(1')}] There exists $R>0$ such that for every $x_0\in M$ we have
$$
\|L_{yx}(\nabla f(y))-\nabla f(x)\|\leq Cd(x,y)
$$ for every $x,y\in B(x_0, R)$.
\end{enumerate}
\end{theorem}

Theorems \ref{main theorem} and \ref{characterizations of Lipschitz gradients} are the main
results of this paper, but let us also mention a couple of auxiliary results that may be
useful in general.

Besides parallel translation, another natural, semiglobal way to compare vectors in different fibers $T_xM$, $T_yM$ of $TM$ with $d(x,y)<i(x)$ is by means of the differential of the exponential map
$$
d\exp_{x}\left(v\right):T(T_xM)_v\equiv T_xM\to T_yM,
$$
where $v=\exp_{x}^{-1}(y)$.
It is a straightforward consequence of the definition of $P$ as a solution to a linear ordinary differential equation with initial condition $P(0)=h$, and of the fact that $d\exp_{x}(0)(h)=h$, that
$$
\lim_{y\to x}\sup_{h\in T_xM, \|h\|=1}|d\exp_{x}\left(\exp_{x}^{-1}(y)\right)(h)-L_{xy}(h)|=0.
$$
That is,
$
\lim_{y\to x}\|d\exp_{x}\left(\exp_{x}^{-1}(y)\right)-L_{xy}\|_{\mathcal{L}\left(T_xM, T_yM\right)}=0.
$
However, in sections 5 and 6 we will need much sharper estimations on the rate of this convergence.
In particular, we will need to use the fact that, locally, one has
$$
\|d\exp_{x}\left(\exp_{x}^{-1}(y)\right)-L_{xy}\|_{\mathcal{L}\left(T_xM, T_yM\right)}=O\left(d(x,y)^{2}\right).
$$
This fact might be known, at least in the finite dimensional case, but we have not been able
to find a reference. Of course there are well known estimates of the form
$$
d\exp_{x}\left(t\frac{v}{\|v\|}\right)(th)-P(th)=O(t^3),
$$
see \cite[Chapter IX, Proposition 5.3]{Lang} for instance,
but we want this kind of estimate to hold locally uniformly with respect to $x,v, h$.
So we provide a proof in Section 4. As a consequence we will also show that
$$
\| d(\exp_{x}^{-1})(y)\circ L_{xy}-I\|_{\mathcal{L}\left(T_xM, T_xM\right)}
=O\left(d(x,y)^{2}\right)
$$
locally uniformly.

In section 3 we establish a convexity lemma which is one of the fundamental
ingredients of the proof that the regularizations $(f_{\lambda})^{\mu}$ are uniformly locally
semiconvex and semiconcave.  If $M$ is a Riemannian manifold of nonpositive sectional curvature $K$ with $i(M)>0$, $c(M)>0$, it is well known that the functions $B(x_0, R)\times B(x_0, R)\ni (x,y)\mapsto d(x,y)^2$ and $B(x_0, R)\ni x\mapsto d(x,x_0)^{2}$ are $C^\infty$ and convex, provided that $2R<\min\{i(M), c(M)\}$.
In Lemma \ref{convexity lemma for nonpositive K} below we will see that when $-K_0\leq K\leq 0$ and $i(M)>0$, $c(M)>0$, these functions are so evenly
convex that their sum, multiplied by a suitable positive number dependent only on $R, K_0$,
can compensate the concavity of the function $y\mapsto -d(y, y_0)^{2}$, in a uniform manner with respect to points $x_0, y_0\in M$ such that $d(x_0, y_0)<R$.

On the other hand, in the general case (for instance if $K>0$) it is not true that the mapping $(x,y)\mapsto d(x,y)$ is locally convex, not even when $(x,y)$ move in an arbitrarily small neighborhood of a point $(x_0, x_0)\in M\times M$.
In this situation it is remarkable that, if one assumes that the sectional curvature $K$
is bounded on $M$ (though not necessarily nonpositive), then one can show that this compensation
property still holds for sufficiently small $R$, depending on the bound for the curvature, but
independent of $x_0, y_0$. This is proved in Lemma \ref{convexity lemma}.

The rest of the paper is organized as follows. In Section 2 we gather several basic properties of the regularizations $f_{\lambda}$
which will be used in the rest of the paper. In Section 5 we prove that if $f:M\to\R$ is
locally $C$-semiconvex and locally $C$-semiconcave then $f\in C^{1,1}(M)$, with
$\textrm{Lip}(\nabla f)\leq 12 C$. In section 6 we trim this estimate down to an optimal $2C$ in the case when
$M$ is finite dimensional, and we prove Theorem \ref{characterizations of Lipschitz gradients}.
In Section 7 we combine all the results of the previous sections to produce a proof of
Theorem \ref{main theorem}.
Finally in Section 8 we show that in Theorem \ref{main theorem} one cannot dispense with
the boundedness assumptions on $f$ and $K$.

Our notation is mostly standard, and we generally refer to Sakai's book \cite{Sakai} for any unexplained terms. In Section 3 we will use the second variation formulae for the energy and length functionals, as well as the Rauch comparison theorems for Jacobi fields. We refer the reader to \cite{Sakai, CheegerEbin} for the finite-dimensional case, or to \cite{Klingenberg, Lang} for the infinite-dimensional case. In both cases, we will nevertheless use the notation of do Carmo's book \cite{doCarmo} for Jacobi fields along geodesics and their derivatives.

\medskip

\section{General properties of inf and sup convolutions}

The following Proposition shows how, under certain conditions, the
inf defining $f_{\lambda}(x)$ can be localized on a neighborhood of
the point $x$. We say that a function $f:M\to\mathbb{R}\cup\{+\infty\}$ is {\em quadratically minorized} provided that there exist $c>0$,
$x_{0}\in M$ such that
$$f(x)\geq -\frac{c}{2}(1+d(x, x_{0})^{2})$$ for all $x\in M$.
\begin{proposition}\label{localization}
Let $M$ be a Riemannian manifold,
$f:M\to\mathbb{R}\cup\{+\infty\}$ be quadratically minorized. Let $x\in M$ be such that $f(x)<+\infty$. Then, for
all $\lambda\in (0,\frac{1}{2c})$ and for all $\rho>\bar{\rho}$,
where $$\bar{\rho}=\bar{\rho}(x, \lambda,
c):=\left(\frac{2f(x)+c(2 d(x,x_{0})^{2}+1)}{1-2\lambda
c}\right)^{1/2},$$ we have that
$$
f_{\lambda}(x)=\inf_{y\in
B(x,\rho)}\{f(y)+\frac{1}{2\lambda}d(x,y)^{2}\}.
$$
Moreover, if $f$ is bounded on $M$, say $|f|\leq N$, then the infimum defining
$f_{\lambda}(x)$ can be restricted to the ball $B(x, 2\sqrt{N\lambda})$.
On the other hand, if $f$ is Lipschitz on $M$, then
the infimum defining
$f_{\lambda}(x)$ can be restricted to the ball
$B\left(x, 2\lambda\textrm{Lip}(f)\right)$.
\end{proposition}
\begin{proof}
The first part is \cite[Proposition 2.1]{AF}. Let us prove the last two statements.
If $|f|\leq N$ and $d(y,x)>2\sqrt{N\lambda}$ then
$$f(y)+\frac{1}{2\lambda}d(x,y)^{2}>-N+2N=N\geq f(x)\geq
f_{\lambda}(x),$$
hence
$$f_{\lambda}(x)=\inf_{y\in B(x, \sqrt{N\lambda})}\{f(y)+\frac{1}{2\lambda} d(x, y)^{2}\}.
$$
On the other hand, if $f$ is Lipschitz and $d(x,y)>2\lambda\textrm{Lip}(f)$ then we have
$$f(y)+\frac{1}{2\lambda}d(x,y)^{2}\geq f(x)-\textrm{Lip}(f) d(x,y)+\frac{1}{2\lambda}d(x,y)^{2}\geq
f(x)\geq f_{\lambda}(x),$$
hence
$$
f_{\lambda}(x)=\inf_{y\in B(x, 2\lambda\textrm{Lip}(f))}\{f(y)+\frac{1}{2\lambda}
d(x, y)^{2}\}.
$$
\end{proof}

The following two propositions were proved in \cite{AF}.

\begin{proposition}
Let $M$ be a Riemannian manifold,
$f, h:M\to\mathbb{R}\cup\{+\infty\}$. We have that:
\begin{enumerate}
\item $f_{\lambda}\leq f$ for all $\lambda>0$.
\item If $0<\lambda_{1}<\lambda_{2}$ then $f_{\lambda_{2}}\leq
f_{\lambda_{1}}$.
\item $\inf f_{\lambda}= \inf f$ and, moreover, if $f$ is lower
semicontinuous then every minimizer of $f_{\lambda}$ is a
minimizer of $f$, and conversely.
\item If $T$ is an isometry of $M$ onto $M$, and $f$ is invariant
under $T$ (that is, $f(Tz)=f(z)$ for all $z\in M$), then
$f_{\lambda}$ is also invariant under $T$, for all $\lambda>0$.
\item If $f\leq h$ then $f_{\lambda}\leq h_{\lambda}$ for every $\lambda>0$.
\end{enumerate}
\end{proposition}

\medskip

\begin{proposition}\label{convergence}
Let $M$ be a Riemannian manifold,
$f:M\to\mathbb{R}\cup\{+\infty\}$ a $c$-quadratically minorized function for some $c>0$.
\begin{enumerate}
\item If $f$ is uniformly continuous and bounded on all
of $M$ then $\lim_{\lambda\to 0}f_{\lambda}=f$ uniformly on $M$.
\item If $f$ is uniformly continuous on bounded subsets
of $M$ then $\lim_{\lambda\to 0}f_{\lambda}=f$ uniformly on each
bounded subset of $M$.
\item In general (that is, under no continuity assumptions on $f$)
we have that $\lim_{\lambda\to 0}f_{\lambda}(x)=f(x)$ for every
$x\in M$ with $f(x)<+\infty$.
\end{enumerate}
\end{proposition}

When $M=\R^n$ or a Hilbert space, it is a well known fact (and easy to prove) that
the operation  $f\mapsto f_{\lambda}$ preserves global or local Lipschitz and
convexity properties of $f$. In the Riemannian setting one has to impose curvature
restrictions on $M$ in order to obtain similar results, see \cite{AF},
and the proofs are somewhat subtler.

In order to see that Lipschitz constants of $f$ are almost preserved by passing to
the regularizations $f_{\lambda}$ or $f^{\mu}$, we will use the following.
\begin{lemma}
Let $M$ be a Riemannian manifold such that $i(M)>0$, $c(M)>0$. Assume that the sectional
curvature $K$ of $M$ is bounded below by $-K_0$ for some $K_0>0$. Then, for every $\varepsilon>0$
there exists $r>0$ such that
$$
d\left(\exp_{x}\left( L_{zx}(w)\right), \exp_{z}(w)\right)\leq (1+\varepsilon) d(x,z)
$$
for every $x,z\in M$ with $d(x,z)\leq r$ and every $w\in T_zM$ with $\|w\|\leq r$.
\end{lemma}
\begin{proof}
By \cite[Corollary 1.31]{CheegerEbin}, it suffices to prove the Lemma in the case when
$M$ is a hyperbolic plane of constant curvature $-K_0<0$ (note that Corollary 1.31 of
\cite{CheegerEbin} is true for infinite-dimensional manifolds as well, since its proof
only relies on the Rauch comparison theorem, which remains true in the infinite dimensional
setting, see \cite{Lang}). But in the two-dimensional case of constant negative curvature,
the Lemma is an exercise which can be solved first locally, by considering an exponential
chart $\exp_{z}$ and applying Gronwall's inequality to the corresponding local expression of the geodesic flow, and
then globally, by using the fact that, for any two given balls of the same radius in a simply connected space of constant curvature, there always exists an isometry mapping one ball onto the other one.
\end{proof}

\begin{proposition}
Let $M$ be a Riemannian manifold such that $i(M)>0$, $c(M)>0$. Assume that the sectional
curvature $K$ of $M$ is bounded below by $-K_0$ for some $K_0>0$. Let
$f:M\to\mathbb{R}$ be a Lipschitz function. Then the functions
$f_{\lambda}$ are also Lipschitz, and
$$
\lim_{\lambda\to 0^{+}}\textrm{Lip}(f_{\lambda})=\textrm{Lip}(f).
$$
\end{proposition}
\begin{proof}
We may assume $\textrm{Lip}(f)>0$ (as $f_{\lambda}$ is constant whenever $f$ is constant).
Given $\varepsilon>0$, let $r=r(\varepsilon, K_0)>0$ be as in the statement of the preceding Lemma. Using Proposition
\ref{localization}, we have that, for $\lambda\in \left(0, r/2\textrm{Lip}(f)\right)$, the
infimum defining $f_{\lambda}(x)$ can be restricted to the ball $B(x, r)$. Then we can write
$$
f_{\lambda}(x)=\inf_{v\in B_{T_xM}(0,r)}\{f(\exp_{x}(v))+\frac{1}{2\lambda}\|v\|^{2}\}
$$
for every $x\in M$ and $\lambda\in \left(0, r/2\textrm{Lip}(f)\right)$. Given $x,z\in M$ with $d(x,z)<r$, for every $\delta>0$ we can find
$w_{\delta,z}\in B_{T_zM}(0,r)$ such that
$$
f(\exp_{z}(w_{\delta,z}))+\frac{1}{2\lambda}\|w_{\delta,z}\|^{2}\leq f_{\lambda}(z)+\delta,
$$ and therefore
\begin{eqnarray*}
& & f_{\lambda}(x)-f_{\lambda}(z)\leq \\
& &f\left(\exp_{x}(L_{zx}(w_{\delta,z}))\right)+
\frac{1}{2\lambda}\|L_{zx}(w_{\delta,z})\|^{2}-f\left(\exp_{z}(w_{\delta,z})\right)
-\frac{1}{2\lambda}\|w_{\delta,z}\|^{2}+\delta=\\
& & f\left(\exp_{x}(L_{zx}(w_{\delta,z}))\right)-f\left(\exp_{z}(w_{\delta,z})\right)+
\delta\leq\\
& &\textrm{Lip}(f)d\left(\exp_{x}\left( L_{zx}(w_{\delta,z})\right), \exp_{z}(w_{\delta,z})\right)\leq
\textrm{Lip}(f)(1+\varepsilon) d(x,z) +\delta.
\end{eqnarray*}
By letting $\delta\to 0^{+}$ we obtain
$$
f_{\lambda}(x)-f_{\lambda}(z)\leq\textrm{Lip}(f)(1+\varepsilon) d(x,z),
$$
and by interchanging the roles of $z$ and $x$ we deduce that
$$
|f_{\lambda}(x)-f_{\lambda}(z)|\leq\textrm{Lip}(f)(1+\varepsilon) d(x,z)
$$
for every $x,z\in M$ with $d(x,z)<r$. This shows that $f_{\lambda}$ is locally
$(1+\varepsilon)\textrm{Lip}(f)$-Lipschitz for every
$\lambda\in \left(0, r/2\textrm{Lip}(f)\right)$, and because $M$ is a Riemannian manifold it follows
that $f_{\lambda}$ is globally $(1+\varepsilon)\textrm{Lip}(f)$-Lipschitz for
$\lambda\in \left(0, r/2\textrm{Lip}(f)\right)$. Hence we have $\limsup_{\lambda\to 0^{+}}\textrm{Lip}(f_{\lambda})\leq\textrm{Lip}(f)$. On the other hand, since $\lim_{\lambda\to 0^{+}}f_{\lambda}(x)=f(x)$ for every $x$, it is immediately checked that $\textrm{Lip}(f)\leq \liminf_{\lambda\to 0^{+}}\textrm{Lip}(f_{\lambda})$.
\end{proof}
An analogous result for locally Lipschitz functions $f$ easily follows from the
preceding Proposition and Proposition \ref{localization}.
We let the reader write the corresponding statement.

Now let us state some results from \cite{AF} concerning convexity properties
of $f_{\lambda}$. The following Lemma will be useful in the proof of Theorem \ref{main theorem}.

\begin{lemma}\label{the partial inf of a jointly convex is convex}
Let $M$ be a Riemannian manifold, and $F:M\times
M\to\mathbb{R}\cup\{+\infty\}$ a convex function (where $M\times
M$ is endowed with its natural product Riemannian metric). Assume
either that $M$ has the property that every two points can be
connected by a geodesic in $M$, or else that $F$ is continuous and
$M$ is complete. Then, the function $\psi:M\to\mathbb{R}$ defined
by
$$\psi(x)=\inf_{y\in M}F(x,y)$$ is also convex. Similarly, if $G:M\times
M\to\mathbb{R}\cup\{-\infty\}$ is a concave function (and under the same assumptions on
$M$ or on continuity of $G$) then the function
$$M\ni x\mapsto\phi(x)=\sup_{y\in M}G(x,y)$$ is concave.
\end{lemma}

\medskip

\begin{definition}\label{definition of uniformly locally convex near the diagonal}
{\em Let $M$ be a Riemannian manifold. We say that the distance
function $d:M\times M\to\mathbb{R}$ is uniformly locally convex on
bounded sets near the diagonal if, for every bounded subset $B$ of
$M$ there exists $r>0$ such that $d$ is convex on $B(x,r)\times
B(x,r)$, and the set $B(x,r)$ is convex in $M$, for all $x\in
B)$.}
\end{definition}
Every complete finite-dimensional Riemannian
manifold of nonpositive sectional curvature satisfies this condition,
as we indicated in \cite{AF}. We conclude this section with the following Proposition from \cite{AF}.

\begin{proposition}\label{convexity preservation}
Let $M$ be a Riemannian manifold with the property that any two
points of $M$ can be joined by a minimizing geodesic, and let
$f:M\to\mathbb{R}\cup\{+\infty\}$ be a lower-semicontinuous convex
function.
\begin{enumerate}
\item  Assume that $f$ is bounded on
bounded sets and that the distance function $d:M\times
M\to\mathbb{R}$ is uniformly locally convex on bounded sets near
the diagonal. Then, for every bounded subset $B$ of $M$ there
exists $\lambda_{0}>0$ such that $f_{\lambda}$ is convex on $B$
for all $\lambda\in (0,\lambda_{0})$.

\item Assume that the distance function $d:M\times
M\to\mathbb{R}$ is convex on all of $M\times M$. Then $f_{\lambda}$
is convex on $M$ for every
$\lambda>0$.
\end{enumerate}
Finally, if one assumes that $f$ is continuous and $M$ is complete, it is
not necessary to require that every two points of $M$ can be
connected by a minimizing geodesic in $M$ in order that the above statements hold true.
\end{proposition}

In particular we see that if $M$ is a Cartan-Hadamard manifold and $f:M\to\R$ is convex then
the functions $f_{\lambda}$ are convex. Under the assumptions of Theorem \ref{main theorem}
it is not difficult to see that then $f_{\lambda}=(f_{\lambda})^{\mu}$ are locally $C^{1,1}$.
This provides a useful regularization method for (not necessarily strongly) convex functions
on such manifolds. See \cite{GreeneWu, A} for more background on such topics.

\medskip

\section{A key convexity lemma}

\begin{lemma}\label{convexity lemma}
Let $M$ be a Riemannian manifold with sectional curvature $K$ such that $-K_0\leq K\leq K_0$
for some $K_0>0$. Assume also that $i(M)>0$ and $c(M)>0$. Let $q>1$. Then:
\begin{enumerate}
\item There exists $R=R(K_0, q)>0$ such that for every $C\geq 0$, for every $A\geq 2 C$ and $B\geq qA$, and for every $x_0\in M$ and $y_0\in B(x_0, R)$, the function
$$
\varphi(x,y):=A d(x,y)^2+B d(x, x_0)^{2}-C d(y, y_0)^{2}
$$ is convex on $B(x_0, R)\times B(x_0, R)$.
\item There also exists $R'=R'(K_0, q)>0$ such that for every $C>0$ and $B\geq qC$, and for every $x_0, y_0, z_0\in M$ with $z_0, y_0\in B(x_0, R')$, the function $$\phi(x):=B d(x, z_0)^{2}-C d(x, y_0)^{2}$$ is convex on $B(x_0, R')$.
\end{enumerate}
\end{lemma}
\begin{proof}
\noindent {\bf I.} Let us first consider the function $B(x_0, R)\times B(x_0, R)\ni (x,y)\mapsto \psi(x,y)=d(x,y)^{2}$, where, for the time being,
\begin{equation}\label{range of Rs}
0< 2R<\min\{i(M), c(M), \pi/4\sqrt{K_0}\}
\end{equation}
(we will impose more restrictions on $R$ later on).
We have to estimate the Hessian $D^{2}\psi(x,y)(v, w)^{2}$.
Let $\gamma$ be the unique minimizing geodesic of speed $1$ connecting the points $x$ and $y$, denote the length of $\gamma$ by $\ell=d(x,y)$, and let $X$ be the unique Jacobi field along $\gamma$ such that $X(0)=v$ and $X(\ell)=w$ (note that the points $x$ and $y$ are not conjugate because $d(x,y)<2R< i(M)$). We have
\begin{eqnarray*}
& & D^{2}\psi(x,y)(v,w)^{2}=2\ell \left( \langle
X(\ell),
    X'(\ell)\rangle - \langle X(0),
    X'(0)\rangle \right)=\\
& &= 2\ell \int_{0}^{\ell}\left( \langle X', X'\rangle -
    \langle R(\gamma',X)\gamma', X\rangle \right)dt,
\end{eqnarray*}
where $R$ is the curvature tensor (defined by $R(X,Y)Z=\nabla_{X}\nabla_{Y}Z-\nabla_{Y}\nabla_{X}Z-\nabla_{[X,Y]}Z$ as in \cite{Sakai}).
In particular, from the second equality, it is obvious that when $M$ has sectional curvature
$K\leq 0$ one has $D^{2}\psi(x,y)(v,w)^{2}\geq 0$, hence $\psi$ is convex on the set $B(x_0,R)\times B(x_0, R)$.

Because of the linearity of the Jacobi equation, the field $X$ can be written as $X=W+V$, where $W$ is the unique Jacobi field along $\gamma$ such that $W(0)=0$, $W(\ell)=w$, and $V$ is the unique Jacobi field along $\gamma$ with $V(0)=v$, $V(\ell)=0$.

\medskip

\noindent {\bf II.} Let us suppose first that the fields $W, V$ are both orthogonal to $\gamma$.
Using the Rauch comparison theorem (as stated, for instance in \cite[Theorem 2.3(b) of Chapter IV, p. 149]{Sakai}, which also holds in the infinite dimensional case, see \cite[Chapter XI, Theorem 5.1 and its proof]{Lang}), we obtain, by comparing the Jacobi field $W$ with a corresponding Jacobi field $Y$ in a space $E$ of constant curvature $K_0$ (for instance a suitable sphere in the Euclidean or the Hilbert space), that
$$
\|w\|=\|W(\ell)\|\geq \|Y(\ell)\|=\frac{\sin\left(\sqrt{K_0}\ell\right)}{\sqrt{K_0}}\|W'(0)\|
$$
because in this case $Y(t)=\frac{\sin\left(\sqrt{K_0}\ell\right)}{\sqrt{K_0}} P(t)$, where $P(t)$ denotes the parallel translation of $W'(0)$ along the corresponding geodesic. Similarly, now comparing $W$ with a corresponding Jacobi field $\tilde{Y}$ in a space of constant curvature equal to $-K_0$ (for instance a suitable hyperbolic space modelled on an open half-space of the Euclidean or the Hilbert space), we get
$$
\|w\|=\|W(\ell)\|\leq\|\tilde{Y}(\ell)\|=\frac{\sinh\left(\sqrt{K_0}\ell\right)}{\sqrt{K_0}} \|W'(0)\|,
$$
because in this case $\tilde{Y}(t)=\frac{\sinh\left(\sqrt{K_0} t \right)}{\sqrt{K_0}}P(t)$.
Therefore we have
\begin{equation}\label{estimate for w}
 \frac{\sin\left(\sqrt{K_0}\ell\right)}{\sqrt{K_0}}\|W'(0)\|_x \leq \|w\|_y \leq \frac{\sinh\left(\sqrt{K_0}\ell \right)}{\sqrt{K_0}} \|W'(0)\|_x.
\end{equation}
In a similar manner one can also see that
\begin{equation}\label{estimate for v}
 \frac{\sin\left(\sqrt{K_0}\ell\right)}{\sqrt{K_0}} \|V'(\ell)\|_y \leq \|v\|_x \leq \frac{\sinh\left(\sqrt{K_0}\ell\right)}{\sqrt{K_0}}  \|V'(\ell)\|_y.
\end{equation}

Now, using once again the Rauch comparison Theorem (by comparing $W$ with a corresponding Jacobi field  $Y$ in a space $E$ of constant curvature $K_0$) we also have that
$$
\frac{\langle W'(t), W(t)\rangle}{\|W(t)\|^2} \geq \frac{\langle Y'(t), Y(t)\rangle}{\|Y(t)\|^2}, \textrm{ and } \|W(t)\|\geq \|Y(t)\|,
$$
where $Y(t)=\frac{\sin\left(\sqrt{K_0} t\right)}{\sqrt{K_0}} P(t)$, with $P$ as above. Therefore
$$
\langle W'(t), W(t)\rangle \geq \langle Y'(t), Y(t)\rangle= \frac{\sin\left(\sqrt{K_0} t\right) \, \cos\left(\sqrt{K_0} t\right)}{\sqrt{K_0}} \|W'(0)\|^2,
$$
which combined with $(\ref{estimate for w})$ yields
$$
2\ell \langle W'(\ell), W(\ell)\rangle_y
\geq
\frac{2\ell \sqrt{K_0} \sin\left(\sqrt{K_0}\ell\right) \cos\left(\sqrt{K_0}\ell\right)}{\sinh^{2} \left(\sqrt{K_0}\ell\right)}\|w\|_{y}^{2}.
$$
In a similar way one checks that
$$
-2\ell \langle V'(0), V(0)\rangle_x\geq \frac{2\ell \sqrt{K_0} \sin\left(\sqrt{K_0}\ell\right) \cos\left(\sqrt{K_0}\ell\right)}{\sinh^{2} \left(\sqrt{K_0}\ell\right)} \|v\|_{x}^{2}
$$
(just note that $V(t)=J(\ell-t)$, where $J$ is the unique Jacobi field along the geodesic
$t\mapsto \gamma(\ell-t)$ joining $y$ to $x$ with $J(0)=0$, $J(\ell)=v$, and therefore
$V'(t)=-J'(\ell-t)$, which accounts for the sign change in the scalar product).

Now, let $r, s, \varepsilon$ be three\footnote{In the proof of the following Lemma the reader will see why here we choose to work with these three numbers instead of just one.} positive numbers such that
\begin{equation}
2>\frac{1+s}{1-\varepsilon}>1.
\end{equation}
Since the three functions $t\mapsto \frac{t\sin t \cos t}{\sinh^2(t)}$, $t\mapsto\frac{t}{\sin t}$ and $t\mapsto\frac{t\cosh t}{\sinh t}$ tend to $1$ as $t\to 0^{+}$ and are continuous and stricly positive on $(0, \frac{\pi}{4}]$, we can find $R>0$ sufficiently small so that, for all $\ell\in (0, 2R]$,
\begin{eqnarray}\label{what smallness of R gives}
& & \frac{\ell \sqrt{K_0} \sin\left(\sqrt{K_0}\ell\right) \cos\left(\sqrt{K_0}\ell\right)}{\sinh^{2} \left(\sqrt{K_0}\ell\right)}\geq (1-\varepsilon)
\\
& & \frac{\ell \sqrt{K_0}}{\sin\left(\sqrt{K_0}\ell\right)}\leq 1+r \\
& & \frac{t\cosh t}{\sinh t}\leq 1+s
\end{eqnarray}
hold together with (\ref{range of Rs}).

Then we have
\begin{equation}\label{estimate for wv}
2\ell \langle W'(\ell), W(\ell)\rangle_y - 2\ell \langle V'(0), V(0)\rangle_x\geq 2(1-\varepsilon)\left(\|w\|_y^{2}+\|v\|_{x}^2\right).
\end{equation}
Thus, by combining (\ref{estimate for w}), (\ref{estimate for v}), (3.5), and
(\ref{estimate for wv}), we obtain
\begin{eqnarray*}
& & D^2 \psi(x,y)(v,w)^2= 2\ell \left( \langle
X(\ell), X'(\ell)\rangle - \langle X(0),  X'(0)\rangle \right)=\\
& & 2\ell\left( \langle W(\ell)+V(\ell), W'(\ell)+V'(\ell) \rangle - \langle W(0)+V(0), W'(0)+V'(0) \rangle \right)=\\
& & 2\ell\langle W(\ell), W'(\ell)\rangle - 2\ell\langle V(0), V'(0)\rangle +
2\ell\langle w, V'(\ell)\rangle -2\ell\langle v, W'(0)\rangle \geq\\
& & 2(1-\varepsilon)\left( \|w\|_y^{2} + \|v\|_{x}^{2}\right) -2\ell\|W'(0)\|_{x} \|v\|_{y} -2\ell \|w\|_y \|V'(\ell)\|_y \geq\\
& & 2(1-\varepsilon)\left( \|w\|_y^{2} + \|v\|_{x}^{2}\right) -4\frac{\ell\sqrt{K_0}}{\sin\left(\sqrt{K_0}\ell\right)}\|w\|_y \|v\|_x\geq \\
& &  2(1-\varepsilon)\left( \|w\|_y^{2} + \|v\|_{x}^{2}\right) -4(1+r)\|w\|_y \|v\|_x,
\end{eqnarray*}
that is
\begin{equation}\label{estimate for dxy}
D^2 \psi(x,y)(v,w)^2\geq 2(1-\varepsilon)\left( \|w\|_y^{2} + \|v\|_{x}^{2}\right) -4(1+r)\|w\|_y \|v\|_x.
\end{equation}

\medskip

\noindent {\bf III.} Let us now suppose that the Jacobi fields $W$, $V$ are tangent to $\gamma$. Then $X=W+V$ is of the form $X(t)=(at+b)\gamma'(t)$ for some $a, b\in\R$; in particular $v=X(0)=b\gamma'(0)$, $w=X(\ell)=(a\ell+b)\gamma'(\ell)$, and $X'(t)=a\gamma'(t)$. Hence also $\|v\|_{x}^{2}=b^{2}$, $\|w\|_{y}^{2}=a^{2}\ell^{2}+b^{2}+2a\ell b$, $\langle L_{xy}(v), w\rangle_{y}=b^{2}+ab\ell$, and therefore
\begin{eqnarray*}
& &D^{2}\psi(x,y)(v,w)^{2}=
2\ell \left( \langle
X(\ell),
    X'(\ell)\rangle - \langle X(0),
    X'(0)\rangle \right)=\\
    & &
    2\ell \left( \langle
a\ell \gamma'(\ell) +b\gamma'(\ell),
    a\gamma'(\ell)\rangle - \langle b\gamma'(0)),
    a\gamma'(0)\rangle \right)=2\ell^2 a^{2}=\\
    & &
    2\left( \|v\|_{x}^{2}+\|w\|_{y}^{2}-2\langle L_{xy}(v), w\rangle_{y}\right)\geq\\
    & &
    2(1-\varepsilon)\left( \|w\|_y^{2} + \|v\|_{x}^{2}\right) -4\langle L_{xy}(v), w\rangle_{y},
\end{eqnarray*}
that is
\begin{equation}\label{estimate of the tangent component}
D^{2}\psi(x,y)(v,w)^{2}\geq 2(1-\varepsilon)\left( \|w\|_y^{2} + \|v\|_{x}^{2}\right)
-4\langle L_{xy}(v), w\rangle_{y}.
\end{equation}

\noindent {\bf IV.} In the general case we have that every Jacobi field $X$ along $\gamma$
can be written in the form
    $
    X=X^{\top}+X^{\bot},
    $
where $X^{\top}$ and $X^{\bot}$ are Jacobi fields along $\gamma$,
$X^{\top}$ and $(X^{\top})'$ are tangent to $\gamma$, and
$X^{\bot}$ and $(X^{\bot})'$ are orthogonal to $\gamma$ (see for instance \cite[Propositions 2.3 and 2.4 of Chapter IX]{Lang}). In
particular $\langle X^{\top}, (X^{\bot})'\rangle=0$ and $\langle
X^{\bot}, (X^{\top})'\rangle=0$. This implies that
    $$
    \langle X'(t), X(t)\rangle=
    \langle (X^{\top})'(t), X^{\top}(t)\rangle+
    \langle (X^{\bot})'(t), X^{\bot}(t)\rangle,
    $$
and therefore, by combining estimates (\ref{estimate for dxy}) and (\ref{estimate of the tangent component}), we obtain
\begin{eqnarray*}
& &D^{2}\psi(x,y)(v,w)^{2}=
2\ell \left( \langle
X(\ell),
    X'(\ell)\rangle - \langle X(0),
    X'(0)\rangle \right)=
    \\
    & &
    2\ell \left( \langle
X^{\top}(\ell),
    (X^{\top})'(\ell)\rangle - \langle X^{\top}(0),
    (X^{\top})'(0)\rangle \right) +\\
    & &+2\ell \left( \langle
X^{\bot}(\ell),
    (X^{\bot})'(\ell)\rangle - \langle X^{\bot}(0),
    (X^{\bot})'(0)\rangle \right)\geq\\
    & &
    2(1-\varepsilon)\left( \|w^{\top}\|_y^{2} + \|v^{\top}\|_{x}^{2}\right) -4\langle L_{xy}(v^{\top}), w^{\top}\rangle_{y} + \\
    & &
    +2(1-\varepsilon)\left( \|w^{\bot}\|_y^{2} + \|v^{\bot}\|_{x}^{2}\right) -4(1+r)\|w^{\bot}\|_y \|v^{\bot}\|_x \geq\\
    & &
2(1-\varepsilon)\left( \|w\|_y^{2} + \|v\|_{x}^{2}\right) -4(1+r)\|w\|_y \|v\|_x,
\end{eqnarray*}
where we have also used the following inequalities
\begin{eqnarray*}
& &\langle L_{xy}(v^{\top}), w^{\top}\rangle_{y} +\|w^{\bot}\|_y \|v^{\bot}\|_x\leq
\|v^{\top}\|_x \|w^{\top}\|_y+\|w^{\bot}\|_y \|v^{\bot}\|_x\leq \\
& &\|v^{\top}+v^{\bot}\|_x \, \|w^{\top}+w^{\bot}\|_y,
\end{eqnarray*}
of which the second one is immediately checked by squaring both sides and observing that
$$
2\|v^{\top}\|_x \|w^{\bot}\|_y \|v^{\bot}\|_x \|w^{\top}\|_y\leq
\|v^{\top}\|_x^{2} \|w^{\bot}\|_y^{2} + \|v^{\bot}\|_x^{2} \|w^{\top}\|_y^{2}.
$$
Therefore we also get in the general case that
\begin{equation}\label{final estimate for dxy}
D^{2}\psi(x,y)(v,w)^{2} \geq
2(1-\varepsilon)\left( \|w\|_y^{2} + \|v\|_{x}^{2}\right) -4(1+r)\|w\|_y \|v\|_x.
\end{equation}

\medskip

\noindent {\bf V.} Now let us consider the function $B(x_0, R)\ni x\mapsto\eta(x):=d(x, x_0)^2$. If we take $y=x_0$, $w=0$ in the above estimation for $D^{2}\psi(x,y)$, we immediately get
\begin{equation}\label{lower estimate for dx}
D^{2}\eta(x)(v)^{2}=D^{2}\psi(x,x_0)(v,0)^{2}\geq 2(1-\varepsilon)\|v\|_{x}^{2}.
\end{equation}

On the other hand, a further application of the Rauch comparison theorem (very similar to what we have already done, by comparing with a corresponding Jacobi field in a space of constant curvature $K_0$; see for instance \cite[Exercise 4 following Lemma 2.9 in Chapter 4]{Sakai}) shows that
\begin{equation}\label{upper estimate for dx}
D^{2}\eta(x)(v)^{2}=D^{2}\psi(x,x_0)(v,0)^{2}\leq 2\frac{\sqrt{K_0} d(x,x_0) \cosh\left(\sqrt{K_0}d(x,x_0)\right)}{\sinh\left(\sqrt{K_0} d(x,x_0)\right)}\|v\|_{x}^{2},
\end{equation}
and using (3.6) and (\ref{lower estimate for dx}) we get
\begin{equation}\label{estimate for dx}
2(1-\varepsilon)\|v\|_{x}^{2}\leq D^{2}\eta(x)(v)^{2}\leq 2(1+s)\|v\|_{x}^{2}.
\end{equation}

\medskip

\noindent {\bf VI.} Now we can proceed with the proof of the Lemma. Consider the function
$$
B(x_0,R)\times B(x_0, R)\ni (x,y)\mapsto \varphi(x,y)=A\psi(x,y)+B\eta(x; x_0)-C\eta(y; y_0),
$$
where we denote $\eta(x; x_0)=d(x, x_0)^{2}$ and $\eta(y; y_{0})=d(y, y_0)^{2}$.

According to our previous estimates, we have
\begin{eqnarray*}
& & D^{2}\varphi(x,y)(v,w)^{2}=\\
& & A D^{2}\psi(x,y)(v,w)^{2} + BD^{2}\eta(x;x_0)(v)^{2}
- C D^{2}\eta(y;y_0)(w)^{2}\geq\\
& & 2A\left( (1-\varepsilon)\left(\|v\|^{2} + \|w\|^{2}\right)-2(1+r)\|v\| \, \|w\|\right)
+2B(1-\varepsilon)\|v\|^{2}-2C (1+s) \|w\|^{2},
\end{eqnarray*}
and we want to find $A,B$ such that the bottom term is positive. Without loss of generality we may and do assume that $C=1$, and then we need to find $A,B>1$ such that
$$
A(1-\varepsilon)\left(\|v\|^{2} + \|w\|^{2}\right)-2A(1+r)\|v\| \, \|w\|
+B(1-\varepsilon)\|v\|^{2}-(1+s) \|w\|^{2}\geq 0
$$
for all $(v,w)\in TM_{(x,y)}$, $(x,y)\in B(x_0, R)\times B(x_0, R)$.
Assume that $$A\geq 2>\frac{1+s}{1-\varepsilon}.$$
We can then write
\begin{eqnarray*}
& & A(1-\varepsilon)\left(\|v\|^{2} + \|w\|^{2}\right)-2A(1+r)\|v\| \, \|w\|
+B(1-\varepsilon)\|v\|^{2}-(1+s) \|w\|^{2}=\\
& &  (1-\varepsilon)(A+B) \|v\|^{2}+ \left( A(1-\varepsilon)-(1+s)\right)\|w\|^2 -2A(1+r)\|v\| \, \|w\|=\\
& & \left( \alpha \|v\| -\beta\|w\|\right)^{2} + (1-\varepsilon)\left(B-\mathcal{B}(A, r,s,\varepsilon)\right)\|v\|^{2},
\end{eqnarray*}
where
$$
\mathcal{B}(A, r,s,\varepsilon):=\frac{ A^{2}\left( (1+r)^{2} -(1-\varepsilon)^{2}\right) + A(1+s)(1-\varepsilon)}{(1-\varepsilon)\left( A(1-\varepsilon)-(1+s)\right)},
$$
$$
\alpha:=\sqrt{(A+\mathcal{B}(A, r,s,\varepsilon))(1-\varepsilon)}, \,\,\, \beta:=\sqrt{(1-\varepsilon)A-(1+s)}, \textrm{ and }
$$
$$
\alpha\beta =A(1+r).
$$
Now, the functions
$$
[2, \infty)\ni A\mapsto h_{\varepsilon}(A):=\frac{\mathcal{B}(A, \varepsilon,\varepsilon,\varepsilon)}{A}=\frac{4\varepsilon A+1-\varepsilon^{2}}{(1-\varepsilon)^{2}A-1+\varepsilon^{2}}
$$
are easily checked to be decreasing and nonnegative for all $\varepsilon\in (0,\frac{1}{4})$, hence
$$
\max_{A\geq 2}h_{\varepsilon}(A)=h_{\varepsilon}(2)=\frac{8\varepsilon +1-\varepsilon^{2}}{2(1-\varepsilon)^{2}A-1+\varepsilon^{2}},
$$
and this quantity converges to $1$ as $\varepsilon$ goes to $0$. Then we can take
$\varepsilon=s=r$ and assume that $\varepsilon$ (and consequently $R$ too) is small enough so that $h_{\varepsilon}(2)\leq q$, and in particular
$$
0\leq \frac{\mathcal{B}(A, \varepsilon, \varepsilon,\varepsilon)}{A}\leq h_{\varepsilon}(A)\leq q \textrm{ for all } A\geq 2.
$$
Therefore, for all $A\geq 2$ and $B\geq qA$ we also have $B\geq \mathcal{B}(A, \varepsilon,\varepsilon,\varepsilon)\geq 0$, and  consequently
$D^{2}\varphi (x,y)(v,w)^{2}\geq 0$ for every $x_0\in M, y_0\in B(x_0, R)$ and
$x,y\in B(x_0, R)$. Thus $(1)$ is proved.

Finally, let us show $(2)$. This is much easier. We have
\begin{eqnarray*}
& & D^{2}\phi(x)(v)^{2}= BD^{2}\eta(x;z_0)(v)^{2}
- C D^{2}\eta(x;y_0)(v)^{2}\geq\\
& & 2B(1-\varepsilon)\|v\|^{2}-2C (1+s) \|v\|^{2},
\end{eqnarray*}
so it is clear that we can choose $\varepsilon, s, R>0$ small enough so that for all $B\geq qC$ we have $$\frac{B}{C}\geq q\geq\frac{1+s}{1-\varepsilon}$$ and consequently $D^{2}\phi(x)(v)^{2}\geq 0$ for every $x_0, y_0, z_0, x\in M$ with $x, y_0, z_0\in B(x_0, R)$.
\end{proof}

\medskip

There are interesting variants of the preceding Lemma. For instance, if we further assume that the sectional curvature of $M$ is nonpositive, one can show that the mentioned compensation property holds semiglobally.

\begin{lemma}\label{convexity lemma for nonpositive K}
Let $M$ be a Riemannian manifold with sectional curvature $K$ such that $-K_0\leq K\leq 0$ for some $K_0>0$. Assume also that $i(M)>0$ and $c(M)>0$, and fix $R$ with $0< 2R<\min\{i(M), c(M)\}$. Then, for every $C_0\geq0$ there exist $A_0, B_0>0$ (dependent only on $K_0$, $R$, and $C_0$) such that, for every $A\geq A_0$ and $B\geq B_0$, and for every $x_0\in M$ and $y_0\in B(x_0, R)$, the function
$$
\varphi(x,y)=A d(x,y)^2+B d(x, x_0)^{2}-C_{0} d(y, y_0)^{2}
$$ is strongly convex on $B(x_0, R)\times B(x_0, R)$.
\end{lemma}
\begin{proof}
Let us first put $A=A_0$ and $B=B_0$.
The proof goes along the same lines (sometimes using Rauch's theorem to compare a Jacobi field $J$ with a corresponding Jacobi field in a space of constant curvature equal to $0$, for instance the Euclidean or the Hilbert space), in order to arrive to the following estimation
\begin{eqnarray*}
& & D^{2}\varphi(x,y)(v,w)^{2}=\\
& & A_0 D^{2}\psi(x,y)(v,w)^{2} + B_0D^{2}\eta(x;x_0)(v)^{2}
- C_0 D^{2}\eta(y;y_0)(w)^{2}\geq\\
& & 2A_0\left( (1-\varepsilon)\left(\|v\|^{2} + \|w\|^{2}\right)-2\|v\| \, \|w\|\right)
+2B_0(1-\varepsilon)\|v\|^{2}-2C_0 N \|w\|^{2},
\end{eqnarray*}
where now $R$ is fixed and not necessarily small (with the only restriction that
$0<2R<\min\{i(M), c(M)\}$); where $N$ (taking the place of $(1+s)$ in the proof of
Lemma \ref{convexity lemma}) is a number depending only on $R, K_0$, and where
$\varepsilon\in (0, 1)$ is neither particularly small, but also a function of $R, K_0$.
(Now we have $r=0$).

We may assume that $C_0 N=1$, and we have
\begin{eqnarray*}
& & 2A_0\left( (1-\varepsilon)\left(\|v\|^{2} + \|w\|^{2}\right)-2\|v\| \, \|w\|\right)
+2B_0(1-\varepsilon)\|v\|^{2}-2 \|w\|^{2}=\\
& & (1-\varepsilon)A_0\left(\|v\|^{2}+\|w\|^{2}\right) +
(1-\varepsilon)(A_0+2B_0)\|v\|^{2} +\\
& & +\left( (1-\varepsilon)A_0 -2\right) \|w\|^{2} -4A_0 \|v\|\,\|w\|=\\
& &(1-\varepsilon)A_0\left(\|v\|^{2}+\|w\|^{2}\right) + \left( \alpha\|v\|-\beta\|w\|\right)^{2},
\end{eqnarray*}
where
$$
\alpha:=\sqrt{(A_0+2B_0)(1-\varepsilon)}, \,\,\, \beta:=\sqrt{(1-\varepsilon)A_0-2}, \textrm{ and } \alpha\beta =2A_0,
$$
which is easily satisfied if for instance we fix $A_0>2/(1-\varepsilon)$ and define
$$
B_0=\frac{1}{2}\left(\frac{4A_0^2}{(1-\varepsilon)\left( (1-\varepsilon)A_0-2\right)}-A_0\right).
$$
For these $A_0, B_0$ we thus have
$$
D^{2}\varphi(x,y)(v,w)^{2}\geq (1-\varepsilon)A_0\left(\|v\|^{2}+\|w\|^{2}\right),
$$
and therefore the function $\varphi$ is strongly convex on $B(x_0, R)\times B(x_0, R)$.

This shows the Lemma in the case when $A=A_0$ and $B=B_0$. For $A\geq A_0$ and $B\geq B_0$
the result follows at once taking into account that when $K\leq 0$ the function
$B(x_0,R)\times B(x_0, R)\ni (x,y)\mapsto d(x,y)^{2}$ is convex for every $R$ with
$0<2R<\min\{i(M), c(M)\}$.
\end{proof}

\medskip

\section{An estimate of the difference between parallel translation and the differential of the exponential map}

\begin{proposition}\label{estimate of the difference between parallel transport and diff of exp}
Let $M$ be a Riemannian manifold (possibly infinite dimensional). For every $x_0\in M$ there exist $r>0$ and $C>0$ such that
$$
\|d\exp_{x}\left(\exp_{x}^{-1}(y)\right)-L_{xy}\|_{\mathcal{L}\left(T_xM, T_yM\right)}
\leq C d(x,y)^{2}
$$
for every $x,y\in B(x_0, r)$.
\end{proposition}
\begin{proof}
Let $x, y$ be two points of $M$ connected by a minimizing geodesic $\gamma:[0, \ell]\to M$ with $\|\gamma'(0)\|=1$, $\ell=d(x,y)$, and assume that there are no conjugate points in $\gamma[0, \ell]$. For each $h\in T_xM\equiv T(T_xM)_{t\gamma'(0)}$, it is well known that the differential of $\exp_x$ on the segment $[0, \ell v]$ is given by
$$
J(t)=d\exp_{x}\left(t v\right)(t h),
$$
where $v=\gamma'(0)$ and $J:[0, \ell]\to TM$ is the unique vector field along the geodesic $\gamma$ satisfying the Jacobi equation
\begin{equation}\label{Jacobi equation}
J''(t)=-R\left(\gamma'(t),J(t)\right)\gamma'(t)
\end{equation}
with initial conditions
$$
J(0)=0, \, J'(0)=h.
$$
We will denote this particular Jacobi field by $J_{x,v,h}(t)$, which we will abbreviate to $J(t)$ when the data $x,v,h$ are understood.
Because the exponential map $(x, v)\mapsto \exp_{x}(v)$ is of class $C^{\infty}$ on an open subset of $TM$, it is clear that the map
$
(x,v,h,t)\mapsto J_{x,v,h}(t)
$
is also of class $C^{\infty}$ wherever it is defined (in particular for all $x\in M$, $v,h\in T_xM$ with $\|v\|\leq 1, \, \|h\|\leq 1$ and $|t|$ sufficiently small depending on $x$).

Let us also consider $P:[0, \ell]\to TM$, the parallel translation of $h$ along $\gamma$ (that is, the unique parallel field along $\gamma$ with $P(0)=h=J'(0)$).
We will denote this particular parallel field by $P_{x,v,h}(t)$ (but again we will abbreviate this expression to $P(t)$ if the point $x$ and the vectors $v,h$ are understood). With the notation we use, we have
$
P_{x, v, h}(t)=L_{x\exp_{x}(tv)}(h).
$
Since $P(t)$ is the solution of a linear ordinary differential equation which depends $C^{\infty}$-wise on the initial data $x,v, h$, it follows from the theorem of differentiability of the flow of an ODE that the mapping
$
(x,v,h,t)\mapsto P_{x,v,h}(t)
$
is of class $C^{\infty}$ wherever it is defined (in particular, by homogeneity of geodesics and parallel translation, on the same set where $J_{x,v,h}(t)$ is defined).

Consider $\Omega=\{\left( (x,v), (y,h), (z,w), t\right)\in TM\times TM\times TM\times \R \, : \, x=y=z\}$, which is a submanifold of $TM\times TM\times TM\times \R$. We will denote the points of $\Omega$ by $(x,v,h,w,t)$ instead of the more cumbersome expression $\left( (x,v), (x,h), (x,w), t\right)$. According to the considerations we have made, the mapping
$$
\Phi(x,v,h,w,t):= \langle J_{x,v,h}(t)- t P_{x,v,h}(t) \, , \, P_{x,v,w}(t) \rangle
$$
is well defined and of class $C^{\infty}$ on an open subset $\mathcal{U}$ of $\Omega$, and $(x_0, 0, 0, 0, 0)\in \mathcal{U}$ for every given point $x_0\in M$.

Now fix $x_0\in M$. By the definition of the topology of $\Omega$ as a submanifold of $(TM)^{3}\times \R$, there exists $R>0$ such that the mapping $\Phi$ is defined and $C^{\infty}$ on a neighborhood of $(x_0, 0,0,0,0)$ of the form $$\mathcal{U}_0=\{(x,v,h,w,t)\in\Omega \, : \,  \max\{d(x,x_0), \|v\|,
\|h\|, \|w\|, |t|\}\leq R\}.$$ Moreover, because the partial derivative $\partial^{3}\Phi/\partial t^{3}$ is continuous, we can assume that this number $R$ is small enough so that
\begin{equation}\label{bound for the third derivative of Phi with respect to t}
\left| \frac{\partial^{3}\Phi}{\partial t^{3}}(x,v,h,w,t) \right|\leq C_0
\end{equation}
for every $(x,v,h,w,t)\in\mathcal{U}_0$, where
$
C_0=1 + \left| \frac{\partial^{3}\Phi}{\partial t^{3}}(x_0,0,0,0,0) \right|.
$
Now observe that
\begin{equation}
\frac{\partial\Phi}{\partial t }(x,v,h,w,t)=
\langle J_{x,v,h}'(t)- P_{x,v,h}(t) \, , \, P_{x,v,w}(t) \rangle,
\end{equation}
and
\begin{equation}
\frac{\partial^2\Phi}{\partial t^2 }(x,v,h,w,t)=
\langle J_{x,v,h}''(t)\, , \, P_{x,v,w}(t) \rangle.
\end{equation}
Since $h=J'(0)=P(0)$, and $J''(0)=-R(\gamma'(0), J(0))\gamma'(0)=0$ (because $J(0)=0$), we immediately check that for $t=0$ we have
\begin{equation}
0=\Phi(x,v,h,w,0)=\frac{\partial\Phi}{\partial t}(x,v,h,w,0)=\frac{\partial^2\Phi}{\partial t^2}(x,v,h,w,0).
\end{equation}
Therefore, using the fundamental theorem of calculus thrice, and plugging (\ref{bound for the third derivative of Phi with respect to t}), we obtain
\begin{equation}
  \Phi(x,v,h,w,t)=
  \int_{0}^{t}\int_{0}^{s}\int_{0}^{\nu}\frac{\partial^{3}\Phi}{\partial \tau^{3}}(x,v,h,w,\tau) d\tau d\nu ds \leq \frac{C_0 t^{3}}{3}.
\end{equation}
Since for every $w_t\in TM_{\exp_{x}(tv)}$ with $\|w_t\|=R$ there exists $w\in T_xM$ with $\|w\|=R$ and $P_{x,v,w}(t)=w_t$, this implies that
$$
\sup_{w_t\in TM_{\exp_x(tv)}, \|w_t\|=R}\langle J_{x,v,h}(t)- t P_{x,v,h}(t) \, , \, w_t \rangle\leq \frac{C_0 t^{3}}{3},
$$
and therefore that
\begin{equation}
\|J_{x,v,h}(t)- t P_{x,v,h}(t)\| \leq \frac{C_0 t^{3}}{3R}.
\end{equation}
Now, for any given $x,y\in B(x_0, R^{2}/2)$ with $x\neq y$, we set
$$
v:=R\frac{\exp_{x}^{-1}(y)}{\|\exp_{x}^{-1}(y)\|}, \,\,\, t:=\frac{d(x,y)}{R},
$$
and we note that $\|v\|=R$, $0<t\leq R$.
For every $h\in T_xM$ with $\|h\|=R$ we then have
\begin{eqnarray*}
  & & \|d\exp_{x}\left(\exp_{x}^{-1}(y)\right)(h)-L_{xy}(h)\|=\frac{1}{t}
  \|d\exp_{x}\left(tv\right)(th)-tL_{x\exp_{x}(tv)}(h)\|=\\
  & & \frac{1}{t}\| J_{x,v,h}(t)-tP_{x,v,h}(t)\|\leq \frac{1}{t}\frac{C_0 t^{3}}{R}=\frac{ C_0 d(x,y)^{2}}{R^{3}},
\end{eqnarray*}
and taking the sup over those $h\in T_xM$ with $\|h\|=R$ we deduce that
\begin{equation}
  \|d\exp_{x}\left(\exp_{x}^{-1}(y)\right)-L_{xy}\|\leq \frac{C_0 d(x,y)^{2}}{R^{4}},
\end{equation}
which yields the inequality in the statement for $C=C_0/R^{4}$, $r=R^{2}/2$.
\end{proof}

\medskip

\begin{corollary}\label{estimate for composition of Lxy and diff exp}
Let $M$ be a Riemannian manifold (possibly infinite dimensional). For every $x_0\in M$ there exist $r>0$ and $C>0$ such that
$$
\| d(\exp_{x}^{-1})(y) \circ L_{xy}-I\|_{\mathcal{L}\left(T_xM, T_xM\right)}\leq C d(x,y)^{2}
$$
for every $x,y\in B(x_0, r)$.
\end{corollary}
\begin{proof}
Let us denote $B_{xy}=d\exp_{x}\left(\exp_{x}^{-1}(y)\right)$. We know that ${B_{xy}}^{-1}=d(\exp_{x}^{-1})(y)$ is continuous with respect to $x,y$, and $B_{x_0 x_0}=I={B_{x_0 x_0}}^{-1}$, hence there exists $r>0$ such that $\|{B_{xy}}^{-1}\|\leq 2$ whenever $x,y\in B(x_0, r)$. We may also assume $r$ is smaller than the $r$ in the statement of the preceding Proposition, so that we also have $\|L_{xy}-B_{xy}\|\leq C d(x,y)^{2}$ for every $x,y\in B(x_0, r)$. Since $L_{xy}$ is an isometry with inverse $L_{yx}$ we then have
\begin{eqnarray*}
& & \| d(\exp_{x}^{-1})(y) \circ L_{xy}-I\|=\|{B_{xy}}^{-1}\circ L_{xy}-I\|=\\
& & \| {B_{xy}}^{-1}\circ \left(L_{xy}-B_{xy}\right) \|\leq
2 \| L_{xy}-B_{xy}\|\leq 2C d(x,y)^{2}
\end{eqnarray*}
for all $x,y\in B(x_0, r)$.
\end{proof}

\medskip

\section{Semiconcavity, semiconvexity, and Lipschitzness of gradients}

\medskip

The next Proposition is well known and tells us that functions which are locally semiconvex and locally semiconcave are continuously differentiable.

\begin{proposition}\label{f is C1}
Let $M$ be a Riemannian manifold, $B\subset M$ an open convex set, and
$f:B\to\mathbb{R}$ a continuous function. Then $f$ is $C^1$ if and only if there exist
two $C^1$ functions $g,h:B\to\mathbb{R}$ such that $f+g$ is convex and $f-h$ is concave.
\end{proposition}
\begin{proof}
We only need to prove the "if" part. We will use some basic facts about Fr\'echet subdifferentials on Riemannian manifolds (we refer the reader to \cite{AFL2} for an introduction to this topic). The function $f+g$ is subdifferentiable since it is convex and continuous, hence
$f=(f+g)-g$ is subdifferentiable too. On the other hand, $f-h$ is superdifferentiable since it is concave
and continuous, and consequently $f=(f-h)+h$ is also superdifferentiable. We deduce that $f$ is differentiable
on $B$. Finally, because $f+g$ is differentiable and convex, we deduce (see \cite[Proposition 3.8]{AF}) that
$f+g$ is $C^1$, and therefore so is $f$.
\end{proof}

As we are about to see, much more is true.
\begin{proposition}\label{semiconvex and semiconvave implies C11}
Let $M$ be a Riemannian manifold.  If a function $f:M\to\R$ is both locally $C$-semiconvex
and locally $C$-semiconcave, then $f\in C^{1,1}(M)$, with
$\textrm{Lip}(\nabla f)\leq 12C$.\footnote{In Theorem
\ref{characterizations of Lipschitz gradients} below we will show that if
$\textrm{dim}(M)<\infty$ then one has
the following sharp estimation: $\textrm{Lip}(\nabla{f})\leq 2C$.}
\end{proposition}
\begin{proof}
Fix $x_0\in M$. By continuity of the curvature tensor, it is clear that the sectional curvature of $M$ is locally bounded, so there exists $R>0$ such that the sectional curvature of $M$ is bounded by some $K_0$ on the ball $B(x_0, 3R)$.  Then, if $\varphi$ denotes the function $\varphi(x)=C d(x, x_0)^{2}$, defined on this ball, we know that $$|D^2\varphi(z)|\leq
2C\nu \left(\sqrt{K_0} d(z,x_0)\right),$$ where $\nu (t)=t\frac{e^t+e^{-t}}{e^t-e^{-t}}$. As $\nu$ is increasing,
a bound for the Hessian of $\varphi$ on the ball $B(x_0, 3R)$ is $2C\nu \left(\sqrt{K_0} 3R\right)$.  Moreover, since this quantity tends to $2C$ as $R$ goes to $0$, we can assume that $R$ is small enough so that
$$
|D^2\varphi(z)|\leq A:=\frac{24}{11}C \,\,\, \textrm{ for every } z\in B(x_0, 3R).
$$
Note that $A$ does not depend on $x_0$, and that $\nabla\varphi$ is $A$-Lipschitz (according
to definition \ref{definition of Lipschitzness of a gradient}) on the ball $B(x_0, 3R)$ (this
is an easy exercise; if in doubt, see the proof that $(7)\implies (1)$ in Theorem \ref{characterizations of Lipschitz gradients}
below).

We may also assume $R$ is small enough so that $2R<\min\{i(x_0), c(x_0)\}$ and the functions
$\varphi+f,\varphi-f :B(x_0,3R)\to \mathbb{R}$
are convex.

We will start by showing that that
\begin{equation}\label{estimation1}
(f+\varphi)(exp_x(h))-2(f+\varphi)(x)+(f+\varphi)(exp_x(-h))\leq A\|h\|^2
\end{equation}
provided that $x\in B(x_0,R)$ and $\|h\|\leq 2R$, $h\in T_xM$. In order to prove this inequality, observe that
$$
0\leq (\varphi+f)(exp_x(h))-2(\varphi+f)(x)+(\varphi+f)(exp_x(-h))
$$
and
$$
0\leq (\varphi-f)(exp_x(h))-2(\varphi-f)(x)+(\varphi-f)(exp_x(-h))
$$
since both $\varphi+f$ and $\varphi-f$ are convex. The second inequality implies
$$
f(exp_x(h))-2f(x)+f(exp_x(-h))\leq \varphi(exp_x(h))-2\varphi(x)+\varphi(exp_x(-h)).
$$
Hence, plugging this into the first inequality, we have
$$
0\leq (\varphi+f)(exp_x(h))-2(\varphi+f)(x)+(\varphi+f)(exp_x(-h))\leq
$$
$$
\leq 2(\varphi(exp_x(h))-2\varphi(x)+\varphi(exp_x(-h)))\leq A\|h\|^2,
$$
since $\varphi$ is convex and
$$
\varphi(exp_x(h))-\varphi(x)-\langle\nabla\varphi(x), h\rangle\leq \frac{1}{2} A\|h\|^2.
$$
Now observe that from (\ref{estimation1}) it follows that
\begin{equation}\label{estimation2}
(f+\varphi)(exp_x(h))-(f+\varphi)(x)-\langle\nabla(f+\varphi)(x), h\rangle\leq A\|h\|^2.
\end{equation}

We proceed to show that $f+\varphi$ is $C^{1,1}$ with $\textrm{Lip}\left(\nabla(f+\varphi)\right)\leq \frac{9}{2}A$. We already know that $f+\varphi$ is $C^1$ by Proposition \ref{f is C1}.

Let $x,y\in B(x_0,r_0)$, and $h\in T_xM$ with $\|h\|\leq 2r_0$, where $r_0$ will be fixed later on.
Let us set $v=\exp_{x}^{-1}(y)$.
The following inequality is a consequence of $f+\varphi$'s convexity:
$$
\langle\nabla(f+\varphi)(y), L_{xy}h\rangle- \langle\nabla(f+\varphi)(x), h\rangle\leq
$$
$$
\leq (f+\varphi)(exp_y(L_{xy}h))-(f+\varphi)(y)-\langle\nabla(f+\varphi)(x), h\rangle=
$$
$$
= \Bigl( (f+\varphi)(exp_y(L_{xy}h)-(f+\varphi)(x)-\langle\nabla(f+\varphi)(x), h+v\rangle \Bigr)
$$
$$
- \Bigl( (f+\varphi)(y)-(f+\varphi)(x)-\langle\nabla(f+\varphi)(x), v\rangle\Bigr).
$$
Since $f+\varphi$ is convex, the expression in the bottom line is less than or equal to $0$.
We deduce that, for $w\in T_xM$ with $exp_x(w)=exp_y(L_{xy}h)$,
$$
\langle\nabla(f+\varphi)(y), L_{xy}h\rangle- \langle\nabla(f+\varphi)(x), h\rangle\leq
$$
$$
\leq (f+\varphi)(exp_y(L_{xy}h))-(f+\varphi)(x)-\langle\nabla(f+\varphi)(x), h+v\rangle=
$$
$$
=(f+\varphi)(exp_{x}(w))-(f+\varphi)(x)-\langle\nabla(f+\varphi)(x), w\rangle+ \langle\nabla(f+\varphi)(x), w-h-v\rangle.
$$
Now, (\ref{estimation2}) and again the convexity of $f+\varphi$ allow us to deduce
$$
\langle\nabla(f+\varphi)(y), L_{xy}h\rangle -\langle\nabla(f+\varphi)(x), h\rangle\leq
$$
$$
\leq A\|w\|^2+\langle\nabla(f+\varphi)(x), w-h-v\rangle\leq
$$
$$
\leq A\|w\|^2+(f+\varphi)(exp_x(w-h-v))-(f+\varphi)(exp_x(0))\leq
$$
\begin{equation}\label{estimation3}
\leq A\|w\|^2+K_1\|w-h-v\|
\end{equation}
where $K_1$ is the Lipschitz constant of $(f+\varphi)\circ exp_x$ on $B(0,8r_0)$.

Our next step is to estimate $\|w\|$ and $\|w-h-v\|$. On the one hand, we have
\begin{equation}\label{estimation of w}
\|w\|=d(x,exp_y(L_{xy}h))\leq d(x,y)+d(y,exp_y(L_{xy}h))=\|v\|+\|h\|
\end{equation}
On the other hand, we claim:
\begin{claim}
There exist $r_0, K_2>0$ such that
$$
\|w-h-v\|\leq K_2 (\|v\|+\|h\|)^3
$$
for every $x,y\in B(x_0,r_0)$ and every $h$ with $\|h\|\leq 2r_0$.
\end{claim}
We put off the proof of the claim.
By (\ref{estimation3}), $(\ref{estimation of w})$ and the Claim, we deduce
$$
\langle\nabla(f+\varphi)(y), L_{xy}h\rangle- \langle\nabla(f+\varphi)(x), h\rangle\leq A(\|h\|+\|v\|)^2+K_1 K_2 (\|h\|+\|v\|)^3.
$$
We may assume that $r_0$ is small enough such that $$8K_1 K_2 r_0\leq\frac{1}{2}A.$$
Suppose now that $\|h\|=\|v\|=d(x,y)\leq 2r_0$. Then, dividing by $\|h\|$ and taking sup in the left term,
we obtain
$$
||L_{yx}\nabla(f+\varphi)(y)-\nabla(f+\varphi)(x)||\leq  4A\|v\|+8K_1 K_2\|v\|^2\leq
$$
$$
\leq 4A\|v\|+\frac{1}{2}A \|v\|=\frac{9}{2}A d(x,y).
$$
We conclude that $\nabla(f+\varphi)$ is $\frac{9}{2}A$-Lipschitz on $B(x_0, r_0)$. Since $\nabla\varphi$ is $A$-Lipschitz on this ball, $x_0$ is arbitrary and $A$ does not depend on $x_0$, it follows that $f\in C^{1,1}(M)$ with $\textrm{Lip}(\nabla f)\leq \frac{11}{2}A=12 C$.

\medskip

It only remains to prove the claim. Let us define a function $\psi _{xy}:T_xM\to T_xM$ by
$\psi _{xy}=exp_x^{-1}\circ exp_y\circ L_{xy}$.
We have
\begin{equation}\label{equality for w h v}
\|w-v-h\|=
\|\psi _{xy} (h)-\psi _{xy}(0)-h\|.
\end{equation}
Let us now define $\phi _x(h,v)=\psi _{xy}(h)-\psi _{xy}(0)-h$, and
$\Phi (s,t)=\phi _x(sh,tv)$.

The function $\Phi$ satisfies $\Phi (s,0)=\Phi (0,t)=0$ and consequently
$\frac{\partial \Phi}{\partial s}(s,0)=\frac{\partial \Phi}{\partial t}(0,t)=0$ for every $s,t$
small enough. This implies
$\frac{\partial ^2\Phi}{\partial s^2}(0,0)=\frac{\partial ^2\Phi}{\partial t^2}(0,0)=0$.
Moreover, by direct calculation, we have
$$
\frac{\partial \Phi}{\partial s}(s,t)=D\psi _{x,y_t}(sh)(h)-h
$$
where $y_t=exp_x(tv)$. Hence
\begin{eqnarray*}
& &
\frac{\partial \Phi}{\partial s}(0,t)=D\psi _{x,y_t}(0)(h)-h=D(exp_x^{-1}\circ exp_{y_t}\circ L_{x,y_t})(0)(h)-h=\\
& &=Dexp_x^{-1}(y_t)\circ Dexp_{y_t}(0)[L_{xy_t}h]-h=Dexp_x^{-1}(y_t)[L_{xy_t}h]-h=\\
& &
=[Dexp_x(tv)]^{-1}(L_{xy_t}h)-h,
\end{eqnarray*}
and using Corollary \ref{estimate for composition of Lxy and diff exp} we deduce that
$$
\frac{\partial ^2\Phi}{\partial s\partial t}(0,0)=
\lim_{t\to 0}\frac{\frac{\partial \Phi}{\partial s}(0,t)-\frac{\partial \Phi}{\partial s}(0,0)}{t}
=\lim_{t\to 0}\frac{1}{t}\frac{\partial \Phi}{\partial s}(0,t)=0.
$$
This implies $\phi _x(0,0)=0$, $D\phi _x(0,0)=0$, and $D^2\phi _x(0,0)=0$.
Hence by Taylor's Formula,
\begin{equation}\label{Taylor formula}
\phi _x(h,v)=\frac{1}{3!}\int_0^1(1-s)^3D^3\phi _x(sh,sv)(h,v)^3ds
\end{equation}
On the other hand, the theorem on the differentiability of the flow of an ODE implies that the mapping $(x, v, h)\mapsto \phi_x(h,v)$ is $C^{\infty}$, hence $D^{3}\phi_x(h,v)$
is continuous in $(x,v,h)$, and in particular is locally bounded. It follows that there exist $K_2, r_0 >0$ such that
\begin{equation}\label{estimate of Taylors reminder}
\frac{1}{3!}||D^3\phi _x(sh,sv)||\leq K_2
\end{equation}
for every $x,y\in B(x_0,r_0)$, $v=\exp_{x}^{-1}(y)$, $\|h\|\leq 2r_0$, and $s\in [0,1]$.
By combining $(\ref{equality for w h v})$, $(\ref{Taylor formula})$ and $(\ref{estimate of Taylors reminder})$ we conclude that
$$
\|w-v-h\|=\|\phi _x(h,v)\|\leq K_2(\|h\|+\|v\|)^3
$$
for every $x,y\in B(x_0,r_0)$, $\|h\|\leq 2r_0$.
\end{proof}

\medskip

\section{What is a $C^{1,1}$ function?}

\medskip

In this section we will prove Theorem \ref{characterizations of Lipschitz gradients}.
\begin{proof}
$(1)\implies (2)$ is obvious, and $(2)\implies (1)$ is an easy exercise.

\noindent $(1)\implies (3)$. Fix $x\in M$. By Proposition \ref{estimate of the difference between parallel transport and diff of exp} there exist $C', r'>0$ so that
$$
\|d\exp_{x}\left(\exp_{x}^{-1}(y)\right)-L_{xy}\|
\leq C' d(x,y)^{2}
$$
for every $x, y\in B(x_0, 3r')$. We have
$
\sup_{y\in B(x_0, 3r')}\|\nabla f(y)\|<\infty
$ by continuity of $\nabla f$.
Let $r>0$ be such that
$$
\|\nabla f(x)-L_{yx}\nabla f(y)\|\leq Cd(x,y)
$$ for every $x,y\in B(x_0, 2r)$.
By taking a smaller $r$ if necessary, we may assume that $$
r\leq\min\{ r', \frac{\varepsilon}{2C'\left(1+\sup_{y\in B(x_0, 3r')}\|\nabla f(y)\|\right)}\}.
$$
Then we have, for every $x\in B(x_0, r)$ and $v\in T_xM$ with $\|v\|\leq r$,
\begin{eqnarray*}
& &\left| f(\exp_{x}(v))-f(x)-\langle \nabla f(x), v\rangle\right|=\\
& &
\left|\int_{0}^{1}\langle \nabla f(\exp_{x}(tv)), d\exp_{x}(tv)(v)\rangle -
\langle \nabla f(x), v\rangle dt\right|\leq\\
& & \left|\int_{0}^{1}\langle L_{\exp_{x}(tv)x}\left(\nabla f(\exp_{x}(tv))\right) - \nabla f(x), v\rangle dt\right|+\\
& &+\int_{0}^{1}\|\nabla f(\exp_{x}(tv))\|\, \|L_{\exp_{x}(tv) x}-d\exp_{x}(tv)\|\, \|v\|dt\leq\\
& &\int_{0}^{1}Ct\|v\|^{2}dt+C' \|v\|^{3} \sup_{y\in B(x_0, 3r')}\|\nabla f(x)\| \leq\frac{C+\varepsilon}{2} \|v\|^{2}.
\end{eqnarray*}
\noindent $(3)\implies (4)$. Let $\varepsilon>0$ and $q>1$ be such that
$$
q\frac{C+\varepsilon}{2}\leq \frac{C'}{2}.
$$
Given $x_0\in M$, if we apply Lemma \ref{convexity lemma}(2) locally (replacing $M$ with a suitable ball of center $x_0$ where the sectional curvature remains bounded), we get an $R'>0$ so that the function
$$
B(x_0, R')\ni y\mapsto \frac{C'}{2} d(y, y_0)^{2}-\frac{C+\varepsilon}{2}d(y,x)^{2}
$$
is convex, for every $x, y_0\in B(x_0,R')$. We may assume that $R'<r$, where $r$ is as in $(3)$. Let us denote $\varphi(y)=\frac{C'}{2} d(y, y_0)^{2}$, and $\psi(y)=\varphi(y)-\frac{C+\varepsilon}{2}d(y,x)^{2}$. By $(3)$ and convexity of $\psi$ on $B(x_0, R')$ we have, for $x,y\in B(x_0, R')$,
\begin{eqnarray*}
& & f(y)-f(x)+\varphi(y)-\varphi(x)\geq\\
& &\langle\nabla f(x), \exp_{x}^{-1}(y)\rangle -\frac{C+\varepsilon}{2}d(x,y)^{2}+\varphi(y)-\varphi(x)=\\
& & \langle\nabla f(x), \exp_{x}^{-1}(y)\rangle +\psi(y)-\psi(x)\geq\\
& &\langle\nabla f(x), \exp_{x}^{-1}(y)\rangle + \langle\nabla \psi(x), \exp_{x}^{-1}(y)\rangle.
\end{eqnarray*}
This implies that for every $x\in B(x_0,R')$, $v\in T_zM$, $\|v\|=1$, the
function $t\mapsto (f+\varphi)(\exp_{x}(tv))$ is supported by an affine function
of $t$ on a small interval around $0$, which in turn means that $f+\varphi$ is locally
convex along geodesic segments contained in $B(x_0, R')$, hence convex on $B(x_0, R')$.
This shows that the function $f$ is locally $\frac{C'}{2}$-semiconvex for every $C'>C$.
The proof that $f$ is locally $\frac{C'}{2}$-semiconcave is completely analogous.

\medskip

\noindent $(4)\implies (1)$. This is, together with $(4)\implies (5)$, the most delicate
part of the proof. All the previous implications, as well as $(5)\implies (1)$, hold for infinite dimensional manifolds as
well, with the same proofs, but now we will have to use Bangert's generalization of Alexandroff's theorem
on twice differentiability of semiconvex functions defined on finite dimensional
Riemannian manifolds.
According to the results of \cite{Bangert}, the locally semiconvexity of $f$ implies that $f$
admits a Hessian almost everywhere on $M$, in the sense that for almost every $x\in M$
there exists a self-adjoint linear operator $H_x:T_xM\to T_xM$ such that
\begin{equation}\label{AlexandroffBangert}
\lim_{v\to 0}\frac{\|L_{\exp_{x}(v)x}\nabla f(\exp_{x}(v))-\nabla f(x)-H_{x}v\|}{\|v\|}=0
\end{equation}
(recall that in our situation $f\in C^{1}$, so $\nabla f$ exists everywhere, and the
subgradients of $f$ reduce to the usual gradient of $f$ at every point). Of course, if $f\in C^{2}(M)$ then this notion of
Hessian coincides with the usual one.
It is easily seen that $(\ref{AlexandroffBangert})$ implies
\begin{equation}
f(\exp_{x}(v))-f(x)-\langle \nabla f(x), v\rangle -\frac{1}{2}\langle H_x(v),v\rangle=
o\left( \|v\|^{2}\right)
\end{equation}
for every $x$ where $(\ref{AlexandroffBangert})$ holds. We will denote $H_{x}=H_{x}(f)$
if the function $f$ is not understood. Bangert also proved that a
semiconvex function $f$ is convex if and only if $H_x(f)\geq 0$ for every $x$ where
$(\ref{AlexandroffBangert})$ holds.

Therefore, if $(4)$ holds then for each $x_0\in M$ there exists $R>0$ such that $f+\varphi$
is convex and $f-\varphi$ is concave on the ball $B(x_0, R)$, where
$\varphi(x)=\frac{C'}{2}d(x,x_0)^{2}$. We may assume that $R$ is small enough so that
the sectional curvature of $M$ is bounded by some positive number $K_0$ on the ball
$B(x_0, R)$, and then we may take
an $r\in (0,R)$ sufficiently small so that
\begin{equation}\label{bound for Hessian of distance in Bangert part of the proof}
\frac{ r\sqrt{K_0}\cosh\left(r\sqrt{K_0}\right)}{\sinh\left(r\sqrt{K_0}\right)} C'
\leq C'+\varepsilon.
\end{equation}
According to Bangert's results we then
have that $H_{x}(f+\varphi)\geq 0$ and $H_{x}(f-\varphi)\leq 0$ for every $x\in B(x_0, r)$,
which implies that $-H_{x}(\varphi)\leq H_{x}(f)\leq H_{x}(\varphi)$, and since
$(\ref{bound for Hessian of distance in Bangert part of the proof})$ provides a bound for
$H_{x}(\varphi)$ on $B(x_0, r)$, we deduce that
\begin{equation}\label{estimation of Hxf in 4 implies 1}
\|H_{x}(f)\|\leq C'+\varepsilon
\end{equation}
for every $x\in\textrm{Diff}^{2}(f)$, where we denote
$\textrm{Diff}^{2}(f)=\{x \, : \, (\ref{AlexandroffBangert})$ \textrm{ holds}\}.
Now, if for a geodesic segment $c(t)=\exp_{x}(tv)$ we have that
$c(t)\in \textrm{Diff}^{2}(f)$ for almost every $t$, then it is easy to see that,
for every $h\in T_xM$ with $\|h\|=1$, if we denote the parallel translation of $h$ along
$c(t)$ by $P_{h}(t)$, the function
$t\mapsto \langle \nabla f(c(t)), P_h(t)\rangle$ is absolutely continuous and
$$
\langle H_{c(t)}(f)(c'(t)), P_h(t)\rangle=\frac{d}{dt}\langle\nabla f(c(t)), P_h(t)\rangle,
$$
which by integration implies that
\begin{eqnarray}\label{integration of hessian}
& &\langle L_{c(1) x}\left(\nabla f(c(1))\right)-\nabla f(c(0)), h\rangle
=\\
& &
\langle\nabla f(c(1)), P_h(1)\rangle-\langle\nabla f(c(0)), P_h(0)\rangle=\\
& &\int_{0}^{1}\langle H_{c(t)}(f)(c'(t)), P_h(t)\rangle dt\leq\int_{0}^{1}\|H_{c(t)}(f)\|\, \|v\|dt
\leq \left(C'+\varepsilon\right)\|v\|,
\end{eqnarray}
hence, by taking sup on those $h$,
\begin{equation}\label{end of estimation of Lipschitzness by Hessian}
\| L_{\exp_{x}(v) x}\left(\nabla f(\exp_{x}(v))\right)-\nabla f(x)\|\leq
\left(C'+\varepsilon\right)\|v\|.
\end{equation}
On the other hand, since $\textrm{Diff}^{2}(f)$ has full measure, it is immediately seen,
by using Fubini's theorem, that for almost every $v$ one has $\exp_{x}(tv)\in
\textrm{Diff}^{2}(f)$ for almost every $t$. Therefore, if our geodesic segment $c(t)$
does not satisfy $c(t)\in \textrm{Diff}^{2}(f)$ for almost every $t$, then we can at least
take a sequence $(v_k)_{k\in\N}\subset T_xM$ such that $v=\lim_{k\to\infty}v_k$ and $c_k(t)=\exp_{x}(tv_k)$ does satisfy
$c_k(t)\in \textrm{Diff}^{2}(f)$ for almost every $t$, hence
$$
\| L_{\exp_{x}(v_k) x}\left(\nabla f(\exp_{x}(v_k))\right)-\nabla f(x)\|\leq
\left(C'+\varepsilon\right)\|v_k\|,
$$
which yields $(1)$ by taking the limit as $k\to\infty$, using the continuity
of $\nabla f$ and $y\mapsto L_{xy}$, and recalling that $\varepsilon>0$ and $C'>C$ are arbitrary.

\noindent $(4)\implies (5)$. In \cite{Bangert},
Bangert proved that if $f$ is convex for some metric $g$ in a manifold
$M$ then $f$ is locally semiconvex for any other metric $\tilde{g}$ in $M$. It follows
that for every $x\in M$ there exists $R>0$ such that the function $F:B_{T_xM}(0,R)\to\R$
defined by $F(u)=f(\exp_{x}(u))$ is semiconvex. Therefore, from what we have seen in
$(4)\implies (1)$ (applied to the manifold $B_{T_xM}(0,R)$), the gradient $\nabla F$ is
Lipschitz on $B_{T_xM}(0,R)$. Then by Rademacher's theorem $\nabla F$
is differentiable almost everywhere on $B_{T_xM}(0,R)$, and we only have to estimate
$\textrm{Lip}(\nabla F)$. The gradient $\nabla f$ is differentiable in Bangert's sense wherever
$\nabla F$ is differentiable in the usual sense. So, if we denote
$\sigma_{w,v}(t)=\exp_{x}(w+tv)$, we have, for every $w$ where $\nabla F$
is differentiable,
\begin{eqnarray*}
& &D^{2}F(w)(v)^{2}=\frac{d^{2}}{dt^{2}}F(w+tv)|_{t=0}=\\
& &\langle H_{x}(f)(\sigma_{w,v}'(0)),
\sigma_{w,v}'(0)\rangle+\langle\nabla f(\sigma_{w,v}(0)),
\nabla_{\sigma_{w,v}'(0)}\sigma_{w,v}'(0)\rangle.
\end{eqnarray*}
Since the function $(w,v)\mapsto \|\nabla_{\sigma_{w,v}'(0)}\sigma_{w,v}'(0)\|$ is continuous and vanishes on $w=0$,
by compactness of $\{v\in T_xM : \|v\|\leq 1\}$ it immediately follows that, given $\varepsilon>0$,
there exists $r>0$ so that
$$
\|\nabla_{\sigma_{w,v}'(0)}\sigma_{w,v}'(0)\|\leq
\frac{\varepsilon}{1+\sup_{y\in B(x_0, R)}\|\nabla f(y)\|}
$$
for all $w\in B_{T_xM}(0,r)$ and every $v\in T_xM$ with $\|v\|\leq 1$.
By combining the two last
chains of inequalities and using $(\ref{estimation of Hxf in 4 implies 1})$,
we get
$$
|D^{2}F(w)(v)^{2}|\leq |\langle H_{x}(f)(\sigma_{w,v}'(0)),
\sigma_{w,v}'(0)\rangle|+\varepsilon\leq C'+2\varepsilon
$$
for almost every $w\in B_{T_xM}(0,r)$, and for every $v\in T_xM$ with $\|v\|\leq 1$.
Hence $\|D(\nabla F)\|\leq C'+2\varepsilon$ almost everywhere on $B_{T_xM}(0,r)$.
Since $\nabla F$ belongs to the Sobolev space $W^{1, \infty}\left(B_{T_xM}(0,r)\right)$, we
conclude that $\textrm{Lip}(\nabla F)=\|D(\nabla F)\|_{\infty}\leq C'+2\varepsilon$,
which shows $(5)$.

\noindent $(5)\implies (6)$ is trivial.

\noindent $(6)\implies (2)$. Using Corollary \ref{estimate for composition of Lxy and diff exp},
we have
\begin{eqnarray*}
& & |\langle L_{yx}\nabla f(y) -\nabla f(x), h\rangle|=\\
& &|\langle\nabla F(\exp_{x}^{-1}(y)),
d(\exp_{x}^{-1})(y)\circ L_{xy}(h)\rangle-\langle \nabla F(0), h\rangle|\leq\\
& & \langle\nabla F(\exp_{x}^{-1}(y)),
d(\exp_{x}^{-1})(y)\circ L_{xy}(h)-h\rangle|+|\langle \nabla F(\exp_{x}^{-1}(y))-\nabla F(0), h\rangle|\leq\\
& &
\|\nabla F(\exp_{x}^{-1}(y))\| \| d(\exp_{x}^{-1})(y)-L_{yx}\|\,\|h\|+
\|\nabla F(\exp_{x}^{-1}(y))-\nabla F(0)\|\,\|h\|\leq\\
& & O(1)\|h\| O\left(d(x,y)^{2}\right)+ (C+\varepsilon)\|h\|d(x,y).
\end{eqnarray*}
By taking sup on $\{h\in T_xM : \|h\|=1\}$ we get
$$
\|L_{yx}\nabla f(y)-\nabla f(x)\|\leq O\left(d(x,y)^{2}\right)+ (C+\varepsilon)d(x,y).
$$
It follows that
$$
\limsup_{t\to 0^{+}}\frac{1}{t}\|L_{\exp_{x}(tv)x}\left(\nabla f(\exp_{x}(tv))\right)
-\nabla f(x)\|\leq C+\varepsilon,
$$
from which $(2)$ is deduced by letting $\varepsilon$ go to $0$.

We have thus proved the equivalence between statements $(1), (2), ..., (6)$. That $(7)$ follows from $(1)$
is an easy exercise, see
\cite[Exercise 5 and Definition 1.5 in Section 1 of Chapter II]{Sakai}. Conversely one can
deduce $(1)$ from $(7)$ by the same argument as in
$(\ref{integration of hessian})-(\ref{end of estimation of Lipschitzness by Hessian})$,
with the advantage that now we do not have to rely on Bangert's theorem, but on the assumption
that $H_{x}(f)=D^{2}f(x)$ exists for every $x$. Thus, it is worth noting that the equivalence $(1)\iff (7)$ holds for infinite
dimensional manifolds $M$ as well, when $f\in C^{2}(M)$.

Assume now that $M$ is of bounded curvature with $i(M)>0$, $c(M)>0$, and let us prove the
equivalence of $(1), ..., (6)$ to $(4')$ and to $(1')$. Obviously we always have $(4')\implies (4)$ and $(1')\implies (1)$.
One can show that $(4')\implies (1')$ by exactly the same argument we used above in $(4)\implies (1)$,
just noticing that $R$ and $r$ are independent of $x_0$ provided we have a global bound $K_0$
for the sectional curvature of $M$. So we only have to prove that $(4)\implies (4')$.
Given $C''>C'>C>0$, we can choose $q>1$ with $C''\geq qC'$, and use Lemma
\ref{convexity lemma}(2) to find $R>0$ so that, for every $x_0\in M$
and $y_{0}\in B(x_0, R)$,
$$
B(x_0, R)\ni x\mapsto \frac{C''}{2}d(x,x_0)^{2}-\frac{C'}{2}d(x, y_{0})^2 \textrm{ is
convex}.
$$
Now, for every $y_0\in B(x_0, R)$, by $(4)$ there exists $r>0$ such that
$f+\frac{C'}{2}d(\cdot, y_0)^{2}$ is convex on $B(y_0, r)$. Therefore the function
$$
f+\frac{C''}{2}d(\cdot, x_0)^{2}=\left(f+\frac{C'}{2}d(\cdot, y_0)^{2}\right)+
\left(\frac{C''}{2}d(\cdot, x_0)^{2}-\frac{C'}{2}d(\cdot, y_{0})^2\right),
$$
being a sum of two convex functions, is convex on $B(y_0, r)$. Since $y_0\in B(x_0, R)$ is arbitrary,
this shows that $f+\frac{C''}{2}d(\cdot, x_0)^{2}$ is locally convex on $B(x_0, R)$, hence convex
on $B(x_0, R)$, for every $x_0\in M$, and since $R$ is independent of $x_0$ this establishes
$(4')$.
\end{proof}

\begin{remark}
{\em From the above proof it is clear that when $M$ is infinite dimensional the implications
$(1)\iff (2)\implies (3)\implies (4)$ and $(5)\implies (6)\implies (1)$ remain true, and any of these
conditions is equivalent to $(7)$ if $f\in C^{2}(M)$. We also have that $(4)$ implies that $f$
is $C^{1,1}(M)$ with $\textrm{Lip}(\nabla f)\leq 6C'$,
by Proposition \ref{semiconvex and semiconvave implies C11}. }
\end{remark}

\medskip

\section{Proof of Theorem \ref{main theorem}}

We start by establishing the local semiconvexity of the regularizations
$(f_{\lambda})^{\mu}$. This is a rather
straightforward consequence of Proposition \ref{localization} and of the easy part
of Lemma \ref{convexity lemma}.

\begin{proposition}\label{semiconvexity}
Let $M$ be a Riemannian manifold with sectional curvature $K$ such that $-K_0\leq K\leq K_0$ for some $K_0>0$, and such that $i(M)>0, c(M)>0$. Let $f:M\to\R$, $h:M\to \R$ be functions such that
$$
f(x)\geq-\frac{c}{2}\left(1+d(x, x_0)^{2}\right), \textrm{ and } h(x)\leq \frac{c}{2}\left(1+d(x, x_0)^{2}\right)
$$
for all $x\in M$ and some $c>0$. Let $q$ be a number with $q>1$. Then we have:
\begin{enumerate}
\item If $f$ is bounded on $M$ then there exists $\lambda_0>0$ (depending on $K_0$, $\|f\|_{\infty}$ and $q$) such that for every $\lambda\in (0, \lambda_0]$ the function $f_{\lambda}$ is uniformly locally semiconcave with constant $B_\lambda=\frac{q}{2\lambda}$. Similarly, if $h$ is bounded on $M$, then there exists $\mu_0>0$ such that $h^{\mu}$ is uniformly $\frac{q}{2\mu}$-locally semiconvex for every $\mu\in (0, \mu_0]$.
\item If $f$ is bounded on bounded subsets of $M$ then for every bounded set $B\subset M$ there exists $\lambda_0>0$ such that the restriction of the function $f_{\lambda}$ to $B$ is uniformly locally $\frac{q}{2\lambda}$-semiconcave for every $\lambda\in (0, \lambda_0]$. A similar statement holds for $h^{\mu}$.
\end{enumerate}
\end{proposition}
\begin{proof}
It will suffice to prove the Proposition for the functions $h^{\mu}$.
Given $q>1$, let us fix an $R=R(q, K_0)>0$ such that $(1)$ and $(2)$ of Lemma \ref{convexity lemma} hold for $R$ (we may assume $R'=R$ in $(2)$ of this Lemma by making the $R$ in $(1)$ smaller, if necessary).
Using Proposition \ref{localization} we can write
$$
h^{\mu}(x)=\sup_{y\in B(x, \sqrt{k\mu})}\{h(y)-\frac{1}{2\mu} d(x, y)^{2}\}
$$
for all $x\in M$, where $k>0$ is a bound for $|f|$ on $M$. Set $$\mu_0= \frac{R^{2}}{4k}.$$
Then, for every $\mu\in (0, \mu_0]$, every $x_0\in M$, and every $x\in B(x_0, R/2)$, we have that
$$
h^{\mu}(x)=\sup_{y\in B(x, \sqrt{k\mu})}\{h(y)-\frac{1}{2\mu} d(x, y)^{2}\},
$$
and, because $B(x, \sqrt{k\mu})\subseteq B(x, R/2)\subseteq B(x_0, R)$, we also have
$$
h^{\mu}(x)=\sup_{y\in B(x_0, R)}\{h(y)-\frac{1}{2\mu} d(x, y)^{2}\} \textrm{ for every } x\in B(x_0, R/2).
$$
On the other hand, according to Lemma \ref{convexity lemma}(2), we have that for every $C_\mu:=\frac{1}{2\mu}>0$ and every $B_{\mu}\geq q C_\mu$ the function
$$
B(x_0, R)\ni x \mapsto B_{\mu} d(x,x_0)^{2}- C_\mu d(x, y)^{2}
$$
is convex for every $y\in B(x_0, R)$. Since the supremum of a family of convex functions is always convex, we then have that the function
$$
x\mapsto \sup_{y\in B(x_0, R)}\{h(y)-\frac{1}{2\mu} d(x, y)^{2}+B_{\mu} d(x, x_0)^{2}\}=h^{\mu}(x)+B_{\mu} d(x, x_0)^{2}
$$
is convex on the ball $B(x_0, R/2)$. This shows $(1)$. The proof of $(2)$ is similar (one just has to use the first part of Proposition \ref{localization} in order to see that the argument can be repeated for $x_0$ moving on a fixed bounded set $B$).
\end{proof}
\begin{remark}
{\em It is clear that in general one has $\lambda_0=\lambda_{0}(q,K_0, \|f\|_{\infty})\to 0$ as
$q\to 1^{+}$ (also as $\|f\|_{\infty}\to\infty$).}
\end{remark}

That the regularizations $(f_{\lambda})^{\mu}$ are also uniformly locally semiconcave is
a subtler fact. The proof relies on Lemma \ref{convexity lemma}(1).

\begin{proposition}\label{semiconcavity}
Let $M$ be a Riemannian manifold with sectional curvature $K$ such that $-K_0\leq K\leq K_0$ for some $K>0$, and such that $i(M)>0, c(M)>0$. Let $f:M\to\R$, $h:M\to \R$ be functions such that
$$
f(x)\geq-\frac{c}{2}\left(1+d(x, x_0)^{2}\right), \textrm{ and } h(x)\leq \frac{c}{2}\left(1+d(x, x_0)^{2}\right)
$$
for all $x\in M$ and some $c>0$. Fix $q>1$. Then we have:
\begin{enumerate}
\item If $f$ is bounded on $M$ then there exists $\lambda_0>0$ (depending on $q$, $\|f\|_{\infty}$ and $K_0$) such that for every $\lambda\in (0, \lambda_0]$ and for every $\mu\in (0, \frac{\lambda}{2q}]$ the function $(f_{\lambda})^{\mu}$ is uniformly locally semiconcave with constant $B_{\mu} =\frac{q}{2\mu}$.
\item If $f$ is bounded on bounded subsets of $M$ then for every bounded set $B\subset M$ there exists $\lambda_0>0$ such that for every $\lambda\in (0, \lambda_0]$  and for every $\mu\in (0, \frac{\lambda}{2q}]$ the restriction of the function $(f_{\lambda})^{\mu}$ to to $B$ is uniformly locally semiconcave with constant $B_{\mu} =\frac{q}{2\mu}$.
\end{enumerate}
Similar statements (replacing semiconcavity with semiconvexity and interchanging the roles of $\lambda, \mu$) hold for the functions $(h^{\mu})_{\lambda}$.
\end{proposition}
\begin{proof}
Let us assume that $f$ is bounded.
Since $\inf f=\inf f_{\lambda}$ and $f_{\lambda}\leq f$, it is clear that $f_{\lambda}$ is bounded as well, and in fact $\|f_{\lambda}\|_{\infty}\leq \|f\|_{\infty}$ for every $\lambda$. If we take $\lambda_{0}=\mu_0$ and $R$ as in the proof of the preceding Proposition (in particular is $R$ as in the statement of Lemma \ref{convexity lemma}(1)), this implies that for every $(\lambda,\mu)\in (0, \lambda_0]\times (c, \lambda_0]$, for every $x_0\in M$, and for every $x\in B(x_0, R/2)$, we have that
$$
h^{\mu}(x)=\sup_{y\in B(x_0, R)}\{f_{\lambda}(y)-\frac{1}{2\mu} d(x, y)^{2}\}.
$$
Now, for every $\lambda\in (0, \lambda_0]$, using the preceding Proposition, we have that the function
$$
B(x_0, R)\ni y \mapsto f_{\lambda}(y)-C_{\lambda} d(y, x_0)^{2}
$$
is concave, where
$$
C_{\lambda}:=\frac{q}{2\lambda}.
$$
According to Lemma \ref{convexity lemma}(1) (taking $C=C_\lambda$, $A=1/2\mu$, $B\geq qA$), for every $\mu>0$ such that
$$
\frac{1}{2\mu}\geq \frac{q}{\lambda}
$$
and for every $B_{\mu}\geq\frac{q}{2\mu}$, the function
$$
B(x_0, R)\times B(x_0, R)\ni (x,y)\mapsto \frac{1}{2\mu} d(x,y)^{2}+B_{\mu} d(x, x_0)^{2}-C_{\lambda} d(y, x_0)^{2}
$$
is convex.
Equivalently, the function
$$
B(x_0, R)\times B(x_0, R)\ni (x,y)\mapsto C_{\lambda} d(y, x_0)^{2} - \frac{1}{2\mu} d(x,y)^{2} - B_{\mu} d(x, x_0)^{2}
$$
is concave. Therefore the function
\begin{eqnarray*}
& &
B(x_0, R)\times B(x_0, R)\ni (x,y)\mapsto
 f_{\lambda}(y) - \frac{1}{2\mu} d(x,y)^{2} - B_{\mu} d(x, x_0)^{2}=\\
& &
\left(f_{\lambda}(y)-C_{\lambda} d(y, x_0)^{2}\right)+\left(C_{\lambda} d(y, x_0)^{2} - \frac{1}{2\mu} d(x,y)^{2} - B_{\mu} d(x, x_0)^{2}\right),
\end{eqnarray*}
being a sum of concave functions, is concave as well, for every $\mu$ with
$$
0<\mu\leq \frac{\lambda}{2q}.
$$
Hence, using Lemma \ref{the partial inf of a jointly convex is convex} (note that the manifold $B(x_0, R)$ does have the property that every two points can be connected by a minimizing geodesic in $B(x_0, R)$, because of the definition of $R$ in the proof of Lemma \ref{convexity lemma}), we deduce that the function
\begin{eqnarray*}
&
B(x_0, R/2)\ni x\mapsto  & \sup_{y\in B(x_0, R)}\{
 f_{\lambda}(y) - \frac{1}{2\mu} d(x,y)^{2} - B_{\lambda} d(x, x_0)^{2}\}=\\
 & &
 (f_{\lambda})^{\mu}(x)-B_{\lambda} d(x, x_0)^{2}
\end{eqnarray*}
is concave, and this concludes the proof of $(1)$. The proof of $(2)$ is similar and we leave it to the reader's care.
\end{proof}

Theorem \ref{main theorem} immediately follows by combining the preceding Propositions
and the results of sections $2$, $5$ and $6$.

\medskip

\section{Two counterexamples}

If $f$ is a quadratically minorized function defined on $\R^n$ or on the Hilbert space,
then it is known that the functions $(f_{\lambda})^{\mu}$ are of class $C^{1,1}$, no
matter whether $f$ is bounded or not, see \cite{AttouchAze}. An examination of
the above proofs reveals that this result remains true for functions $f$ defined on a
flat Riemannian manifold. However, if $K\neq 0$, in order to obtain $C^{1,1}$ smoothness
of the functions $(f_{\lambda})^{\mu}$, one has to require that both $f$ and $K$ be
bounded (as we did in the statement of Theorem \ref{main theorem}). We next present some
examples showing why this is so.

Let us first see that, even on Cartan-Hadamard manifolds with constant curvature (that is to say, hyperbolic spaces), one cannot dispense with the boundedness assumption on $f$.
\begin{example}\label{counterexample on hyperbolic space}
{\em Let us take $M=\mathbb{H}^{n}$, the hyperbolic space of constant curvature equal to $-1$, modelled on the upper half-space of $\R^n$, with $n\geq 2$.
Let $f:H\to\R$ be defined by
$$
f(x)=d(x,x_0)^{2},
$$
where $d$ denotes the Riemannian distance in $\mathbb{H}^{n}$ and $x_0\in \mathbb{H}^{n}$ is a given point.
The function $f$ is bounded below by $0$, and in particular quadratically minorized. It is also clear that $f$ is uniformly continuous on bounded subsets of $\mathbb{H}^{n}$. We will calculate the functions $(f_{\lambda})^{\mu}$ in this case and see that they are not $C^{1,1}(\mathbb{H}^{n})$.

The function $\mathbb{H}^{n}\ni y\mapsto h(y):=d(y, x_0)^{2}+\frac{1}{2\lambda}d(x, y)^{2}$ is $C^{\infty}$ (because $\mathbb{H}^{n}$ is a Cartan-Hadamard manifold). One can easily see that $\nabla h(y_x)=0$ if and only if $y_x$ is in the geodesic connecting $x$ to $x_0$ and
$$
d(x,y_x)=\frac{\lambda}{1+\lambda}d(x, x_0).
$$
Taking into account the behavior of $h$ at infinity, we infer that
$$\inf_{y\in H}\{d(y, x_0)^{2}+\frac{1}{2\lambda}d(x,y)^2\}=
d(y_x, x_0)+\frac{1}{2\lambda}d(x, y_x)^2.$$ Therefore
$$
f_{\lambda}(x)=h(y_x)=\left(\frac{1}{(1+\lambda)^{2}}+\frac{\lambda}{2(1+\lambda)^{2}}\right)
d(x, x_0)^{2}=\frac{2+\lambda}{2(1+\lambda)^{2}} d(x, x_0)^{2},
$$
which can also be written
$$
f_{\lambda}(x)=\frac{1}{2\lambda'} d(x, x_0)^{2}
$$
for a suitable number $\lambda'>0$.

Similarly, if one considers the function $\mathbb{H}^{n}\ni z\mapsto \psi(z)=\frac{1}{\lambda'}d(z, x_0)^{2}-\frac{1}{2\mu}d(x, z)^{2}$ one can see that
$\nabla\psi(z_x)=0$ exactly when $z_x$ is in the geodesic passing through $x$ and $x_0$, and $\lambda' d(z_x, x) =\mu d(z_x, x_0)$. Taking into account the behaviour of $\psi$ at infinity one can also deduce that
$$
(f_{\lambda})^{\mu}(x)=\sup_{z}\psi(z)=\psi(z_x)=\frac{1}{2(\lambda'-\mu)}d(x, x_0)^{2}.
$$
We do not care about a more explicit expression for $(f_{\lambda})^{\mu}$; the only interesting point is that $(f_{\lambda})^{\mu}= C_{\lambda, \mu} f$ for some positive constant $C_{\lambda, \mu}$.

Therefore it is clear that if $(f_{\lambda})^{\mu}$ were of class $C^{1,1}$ then so would be the square of the distance function, $x\mapsto d(x,x_0)^{2}=f(x)$.
But, in the case of the hyperbolic space $\mathbb{H}^{n}$ one has the following explicit formula for
the Hessian of the square of the distance to a point $x_0$:
$$
D^{2}f(x)(v)^{2}=2\|v\|^{2}\left(\frac{ d(x, x_0)\cosh\left(d(x, x_0)\right)}{\sinh\left( d(x, x_0)\right)}\right)
$$
(see \cite{Sakai} for instance).
Now, because
$
\lim_{t\to\infty}\frac{t\cosh t}{\sinh t}=\infty,
$
it follows that
$
\lim_{d(x, x_0)\to\infty}\|D^{2}f(x)\|=\infty,
$
and since the Hessian of $f$ is unbounded on $\mathbb{H}^{n}$, the gradient of $f$ cannot be Lipschitz on $\mathbb{H}^{n}$.
}
\end{example}

\medskip

Now we will construct an example showing that, even if $f:M\to\R$ is bounded, one has to require that the sectional curvature $K$ of $M$ be bounded, in order that $(f_{\lambda})^{\mu}$ be of class $C^{1,1}$ globally.

\begin{example}
{\em Let $M$ be the half-space of $\R^2$ given by $\{(x_1,x_2)\in\R^2 : x_2>0\}$, with the metric
$$
g_{ij}(x_1, x_2)=\frac{\delta_{ij}}{{x_2}^{4}}.
$$
It is not difficult to show that the curvature of $M$ at a point $p=(x_1, x_2)$ is given by
$
K_p=-2{x_{2}}^{2},
$
and using this fact one can also check that there exists a sequence $(p_n)\subset M$ such that $d(p_n, p_m)\geq 4$ for $n\neq m$ and $K_p\leq -4n^2$ for every $p\in B(p_n, 1)$. Now let us define a function $f:M\to [0,2]$ by
$$
f(p)=\min\{2, \inf_{n\in\N} d(p, p_n)\}.
$$
The function $f$ is obviously bounded and $1$-Lipschitz. Now, the calculation of $((d(\cdot, x_0)^{2})_{\lambda})^{\mu}$ that we carried out in the preceding example works in any
Cartan-Hadamard manifold, hence one can use this fact and Proposition \ref{localization} to see that there exists some $\lambda_0>0$ such that for every $\lambda\in (0, \lambda_0]$ and $\mu\in (0, \lambda)$ there exists a number $C_{\lambda, \mu}>0$ such that
$$
(f_{\lambda})^{\mu}(p)=C_{\lambda,\mu}d(p, p_n)^{2} \textrm{ for every } p\in B(p_n, 1).
$$
Using \cite[Exercise 4 following Lemma 2.9 in Chapter IV, p. 154]{Sakai}, we get that
$$
\|D^{2}(f_{\lambda})^{\mu}(p)\|=\sup_{\|v\|=1}\|D^{2}(f_{\lambda})^{\mu}(p)(v)^2\|\geq
C_{\lambda,\mu}\frac{2n d(p, p_n)\cosh\left(2n d(p,p_n)\right)}{\sinh\left(2n d(p, p_n)\right)}
$$
for every $p\in B(p_n, 1)$. Taking $q_n\in B(p_n, 1)$ with $d(q_n, p_n)=1/2$ we have
$$
\lim_{n\to\infty}\|D^{2}(f_{\lambda})^{\mu}(q_n)\|\geq\lim_{n\to\infty}C_{\lambda, \mu} \frac{n\cosh\left(n\right)}{\sinh\left(n\right)}=\infty,
$$
hence $\|D^{2}(f_{\lambda})^{\mu}\|$ is unbounded on $M$ and consequently $(f_{\lambda})^{\mu}\notin C^{1,1}(M)$.
}
\end{example}



\begin{thebibliography}{}

\bibitem{AnguloVelasco}
J. Angulo and S. Velasco-Forero, {\em Mathematical morphology for real-valued images
on Riemannian manifolds.} In Proc. of ISMM'13 (11th International
Symposium on Mathematical Morphology), Springer LNCS 7883, p. 279--291, 2013.

\bibitem{AttouchAze}
H. Attouch, D. Az\'e, {\em Approximation and regularization of arbitrary functions in Hilbert spaces by the Lasry-Lions method}. Ann. Inst. H. Poincar'e Anal. Non Lin\'eaire 10 (1993), no. 3, 289--312.

\bibitem{A}
D. Azagra, {\em Global and fine approximation of convex functions}. Proc. Lond. Math. Soc.
(3) 107 (2013), no. 4, 799--824.

\bibitem{AF} D. Azagra and J. Ferrera, {\em Inf-convolution and regularization of convex functions  on Riemannian manifolds of nonpositive curvature}. Rev. Mat. Complut. 19 (2006), no. 2, 323--345.

\bibitem{AFL2} D. Azagra, J. Ferrera, F. L\'{o}pez-Mesas, {\em
Nonsmooth analysis and Hamilton-Jacobi equations on Riemannian
manifolds}, J. Funct. Anal. 220 (2005) no. 2, 304--361.

\bibitem{Bangert}
V. Bangert, {\em Analytische Eigenschaften konvexer Funktionen auf
Riemannschen Mannigfaltigkeiten}, J. Reine Angew. Math. 307/308
(1979), 309--324.

\bibitem{Bernard}
P. Bernard, {\em Existence of $C^{1,1}$ critical sub-solutions of the Hamilton-Jacobi equation on compact manifolds}. Ann. Sci. \'Ecole Norm. Sup. (4) 40 (2007), no. 3, 445--452.

\bibitem{CheegerEbin}
J. Cheeger, D.G. Ebin, {\em Comparison theorems in Riemannian geometry.} North-Holland
Mathematical Library, Vol. 9. North-Holland Publishing Co.,
 Amsterdam--Oxford; American Elsevier Publishing Co., Inc., New York, 1975.

\bibitem{doCarmo}
M.P. do Carmo, {\em Riemannian geometry}.  Mathematics: Theory and Applications. Birkh\"auser Boston, Inc., Boston, MA, 1992.

\bibitem{Fathi}
A. Fathi, {\em Regularity of $C^{1}$ solutions of the Hamilton-Jacobi equation}. Ann. Fac. Sci. Toulouse Math. (6) 12 (2003), no. 4, 479--516.

\bibitem{FathisBook}
A. Fathi, {\em Weak KAM Theorem in Lagranian Dynamics}. Book to appear.

\bibitem{GreeneWu}
R. E. Greene, and H. Wu, {\em $C\sp{\infty }$ convex functions and
manifolds of positive curvature}, Acta Math. 137 (1976), no. 3-4,
209--245.

\bibitem{Klingenberg}
W.P.A. Klingenberg, {\em Riemannian geometry}. Second edition. de Gruyter Studies in Mathematics, 1. Walter de Gruyter \& Co., Berlin, 1995.

\bibitem{Lang}
S. Lang, {\em Fundamentals of Differential Geometry}, Graduate
Texts in Mathematics, 191. Springer-Verlag, New York, 1999.

\bibitem{Lasry-Lions}
J.-M. Lasry, and P.-L. Lions,  {\em A remark on regularization in
Hilbert spaces.} Israel J. Math. 55 (1986), no. 3, 257--266.

\bibitem{Sakai}
T. Sakai, {\em Riemannian geometry}. Translations of Mathematical Monographs, 149. American Mathematical Society, Providence, RI, 1996.

\bibitem{Sasaki}
S. Sasaki, {\em On the differential geometry of tangent bundles of Riemannian manifolds}.
T\^ohoku Math. J. (2) 10 (1958) 338--354.

\end{thebibliography}
\end{document}